\documentclass[10pt,a4paper]{amsart}
\usepackage{verbatim}

\usepackage[T1]{fontenc}
\usepackage{graphicx}
\usepackage{enumerate}
\usepackage{amsmath,amsfonts,amssymb}
\usepackage{color}
\usepackage{cite}
\definecolor{citation}{rgb}{0.2,0.58,0.2} 
\definecolor{formula}{rgb}{0.1,0.2,0.6}
\definecolor{url}{rgb}{0.3,0,0.5}
\usepackage{pgf,tikz}
\usepackage{mathrsfs}
\usepackage{fancyhdr}
\usepackage{dutchcal}

\usepackage{dsfont}

\newcommand{\reqnomode}{\tagsleft@false}

\usepackage[colorlinks,pdfpagelabels,pdfstartview = FitH,bookmarksopen = true,bookmarksnumbered = true,urlcolor=url,linkcolor = formula,plainpages = false,hypertexnames = false,citecolor = citation] {hyperref}

\def\dx{\,{\rm d}x}
\def\dy{\,{\rm d}y}
\def \d{\,{\rm d}}
\def \diver{\,{\rm div}}
\def\dist{\,{\rm dist}}

\def\supp{\,{\rm supp}}
\def\diam{\,{\rm diam}}

\allowdisplaybreaks
\makeatletter
\DeclareRobustCommand*{\bfseries}{%
  \not@math@alphabet\bfseries\mathbf
  \fontseries\bfdefault\selectfont
  \boldmath
}

\makeatother

\newlength{\defbaselineskip}
\newcommand{\setlinespacing}[1]
           {\setlength{\baselineskip}{#1 \defbaselineskip}}
\newcommand{\mint}{\mathop{\int\hskip -1,05em -\, \!\!\!}\nolimits}

\newtheorem{theorem}{Theorem}

\newtheorem{definition}{Definition}
\newtheorem{remark}{Remark}[section]
\newtheorem{lemma}{Lemma}[section]
\newtheorem{proposition}{Proposition}[section]
\numberwithin{equation}{section}
\allowdisplaybreaks[4]

\def\en{\mathbb N}
\def\er{\mathbb R}

\newcommand\eps\varepsilon
\def\eqn#1$$#2$${\begin{equation}\label#1#2\end{equation}}

\newcommand{\htt}{\tilde{H}}
\newcommand{\be}{\begin{equation}}

\newcommand{\ee}{\end{equation}}

\newcommand{\rr}{\varrho}

\newcommand{\const}{\operatorname{const}}
\newcommand{\snr}[1]{\lvert #1\rvert}

\newcommand{\nr}[1]{\lVert #1 \rVert}

\newcommand{\N}{\mathbb{N}}

\def\name[#1, #2]{#1 #2}

\def\loc{\operatorname{\rm{loc}}}

\title[]{Regularity results for a class of non-autonomous obstacle problems with $(p,q)$-growth}
\author{Cristiana De Filippis}  \address{Cristiana De Filippis\\Mathematical Institute, University of Oxford\\ Andrew Wiles Building, Radcliffe Observatory Quarter, Woodstock Road, Oxford, OX26GG, Oxford, United Kingdom}
\begin{document}

\subjclass[2010]{35J60, 35J70\vspace{1mm}} 

\keywords{Regularity, non-autonomous functionals, obstacle problems, $(p,q)$-growth\vspace{1mm}}

\thanks{{\it Acknowledgements.}\ This work is supported by the Engineering and Physical Sciences Research Council (EPSRC): CDT Grant Ref. EP/L015811/1. 
\vspace{1mm}}

\maketitle

\begin{abstract}
We study some regularity issues for solutions of non-autonomous obstacle problems with $(p,q)$-growth. Under suitable assumptions, our analysis covers the main models available in the literature.
\end{abstract}
\vspace{3mm}
{\small \tableofcontents}

\setlinespacing{1.08}
\section{Introduction}
Regularity for local minimizers of the functional
\begin{flalign}\label{101}
W^{1,p}(\Omega)\ni w\mapsto \mathcal{F}(w,\Omega):= \int_{\Omega}F(x,Dw) \ \dx
\end{flalign}
where the integrand $F$ has power growth 
\begin{flalign}\label{100}
\snr{z}^{p}\lesssim F(x,z)\lesssim (1+\snr{z}^{2})^{\frac{p}{2}}\qquad p\in (1,\infty)
\end{flalign}
has been investigated in the fundamental works \cite{acfu,giagiudiff,giamod,ha,kumi,kumi1,laur,uh,ur}. The outcome is $C^{0,\beta_{0}}$-regularity for the gradient of solutions with $\beta_{0}\in (0,1)$, and such result is optimal, in the light of the counterexample contained in \cite{ur}. Later on, in the seminal papers \cite{ma2,ma3,ma4,ma1}, was introduced the so-called $(p,q)$-growth condition, i.e.:
\begin{flalign}\label{102}
\snr{z}^{p}\lesssim F(x,z)\lesssim (1+\snr{z}^{2})^{\frac{q}{2}}\qquad 1<p\le q<\infty,
\end{flalign}
which is more flexible than \eqref{100} and allows dealing with models coming from fluid mechanics and material science, \cite{z1,z2,z3}, such as
\begin{flalign*}
w\mapsto \int_{\Omega}\snr{Dw}^{p(x)} \ \dx \qquad \mbox{and}\qquad w\mapsto \int_{\Omega}\left[\snr{Dw}^{p}+a(x)\snr{Dw}^{q}\right] \ \dx.
\end{flalign*}
This new framework has been object of intense investigation over the last two decades, see \cite{acmi,bacomi,bemi,cakrpa,demi,distve,elmama,eslemi,fomami,fumi,haok} for an incomplete list of relatively recent contributions and \cite{dark} for a reasonable survey. In these works is studied the regularity for minimizers of variational integrals like the one in \eqref{101} with \eqref{102} in force, which are "free", in the sense that no additional constraint is imposed on solutions and competitors. Classical examples of constrained variational problems are those involving manifold valued maps, see \cite{de1,de,demi1} for the $(p,q)$-growth case, and obstacle problems. The latter were treated at length in the literature, see \cite{ch,chle,elpa,fu,fuli,hekima,muzi} for variational inequalities modelled upon the $p$-laplacean energy and \cite{bylizh,byleohpa,chde,ele-2,fush,fumi} for more general structures. The underlying principle is that solutions of the obstacle problem should reflect the regularity of the obstacle itself. This holds \emph{verbatim} for linear problems, in which solutions are as regular as the obstacle and for certain nonlinear models with Harnack inequalities and full regularity available for unconstrained minimizers. However, this is no longer the case in the nonlinear setting for general integrands without any specific structure. In this situation, extra regularity must be imposed on the obstacle to balance, in some sense, both the nonlinearity and the non-standard growth. The increasing interest towards the regularity for solutions of obstacle problems is also justified by the fact that they can be employed as comparison maps in the investigation of fine properties of solutions of some PDE, see \cite{chde,fush, hekima,kizo} and references therein.\\
In this paper we provide some regularity results for solutions of non-autonomous obstacle problems with $(p,q)$-growth. In dealing with this, the first big problem arising is the possible occurrence of the Lavrentiev phenomenon, i.e.:
\begin{flalign}\label{103}
\inf_{w\in (W^{1,p}\cap\{w\ge \psi\})}\int_{\Omega}F(x,Dw) \ \dx<\inf_{w\in (W^{1,q}\cap\{w\ge \psi\})}\int_{\Omega}F(x,Dw) \ \dx.
\end{flalign}
This is a clear obstruction to regularity, since \eqref{103} prevents minimizers to belong to $W^{1,q}$. Notice that \eqref{103} cannot happen if $p=q$ or if $F$ is autonomous and convex. Moreover, as pointed out in \cite[Section 3]{eslemi}, the appearance of \eqref{103} has geometrical reasons and cannot be spotted via standard techniques. Therefore, the basic strategy consists in excluding the occurrence of \eqref{103} by imposing that the Lavrentiev gap functional vanishes on solutions: at this stage, the closeness condition formulated in \eqref{pq} below assures the validity of certain a priori estimates, then, a convergence argument renders 
\begin{theorem}\label{t1}
Under assumptions \eqref{104}, \eqref{assf} and \eqref{pq}, let $\psi$ be as in \eqref{obbd} and $g$ as in \eqref{op3}. If the solution $v\in \mathcal{K}_{\psi,g}(\Omega)$ of problem \eqref{op} satisfies \eqref{nolavq}, then it has the following regularity features:
\begin{itemize}
    \item[-] $Dv\in L^{d}_{loc}(\Omega,\mathbb{R}^{n})$ for all $d\in \left[1,\frac{np}{n-\alpha}\right)$;
    \item[-]$V_{\mu,p}(Dv)\in W^{2,\beta}_{loc}(\Omega,\mathbb{R}^{n})$ for all $\beta\in \left(0,\frac{\alpha}{2}\right)$.
\end{itemize}
In particular, if $B_{\rr}\Subset \Omega$ is any ball, there holds that
\begin{flalign}\label{18}
\nr{Dv}_{L^{d}(B_{\rr/2})}\le \frac{c(\texttt{data}_{\texttt{q}},d)}{\rr^{\theta}}\left[1+\int_{B_{\rr}}\left[F(x,Dv)+\snr{D\psi}^{q}\right] \ \dx\right]^{\bar{\gamma}},
\end{flalign}
where $\theta=\theta(n,p,q,\alpha)$ and $\bar{\gamma}=\bar{\gamma}(n,p,q,\alpha,d)$.
\end{theorem}
It is reasonable to expect that, strengthening the regularity assumptions on both, integrand and obstacle, we can actually show better regularity properties than those obtained in Theorem \ref{t1}. In fact, 
\begin{theorem}\label{t2} Under assumptions \eqref{104}, \eqref{assfh}, \eqref{assh} and \eqref{pqh}, let $\psi$ be as in \eqref{obsh}-\eqref{obsh1} and $g$ as in \eqref{op3}. If the solution $v\in \mathcal{K}_{\psi,g}(\Omega)$ of problem \eqref{op} satisfies \eqref{nolavq}, then 
\begin{flalign*}
v\in W^{1,\infty}_{loc}(\Omega).
\end{flalign*}
Moreover, if $B_{\rr}\subset B_{r}\Subset \Omega$ are concentric balls, the following local Lipschitz estimate holds
\begin{flalign*}
\sup_{x\in B_{\rr}}\snr{Dv(x)}\le \left(\frac{c}{r-\rr}\right)^{\tilde{\theta}}\left[1+\int_{B_{r}}F(x,Dv) \ \dx\right]^{\theta},
\end{flalign*}
with $c=c(\texttt{data}_{\infty})$, $\theta=\theta(n,p,q,s)$ and $\tilde{\theta}=\tilde{\theta}(n,p,q,s)$.
\end{theorem}
The Lipschitz bound in Theorem \ref{t2} is essentially realized in three steps: first, the problem is linearized via the identification of a non-negative Radon measure which turns the variational inequality naturally associated to a regularized version of \eqref{op} into an integral identity. Then, the revisited Moser's iteration introduced in \cite{demi} leads to a uniform bound on the sup-norm of the gradient of a suitable sequence of maps approximating the original solution. Finally, careful convergence arguments give the conclusion. \\
The paper is organized as follows: in Section \ref{pre} we list some basic assumptions which will always be in force and  strengthened when needed; well-known results on fractional Sobolev spaces and some useful miscellanea. We also briefly discuss existence and uniqueness for solutions of problem \eqref{op}. In Section \ref{re} we tackle the question of relaxation of functionals with $(p,q)$-growth with obstacle constraint. Sections \ref{pt1}-\ref{pt2} are devoted to the proof of Theorems \ref{t1}-\ref{t2} respectively, while in Section \ref{diff} we provide a higher weak differentiability result for local minimizers of variational integrals with standard $q$-growth and obstacle constraint.
\section{Preliminaries}\label{pre}
\subsection{Main assumptions} \label{ma} In this section we shall collect some minimal hypotheses which will be eventually strengthened throughout the paper. We assume that $\Omega\subset \mathbb{R}^{n}$, $n\ge 2$, is an open, bounded domain with $C^{1}$ boundary and $F\colon \Omega\times \mathbb{R}^{n}\to \mathbb{R}$ is a Carath\'eodory integrand satisfying, for all $x,x_{1},x_{2}\in \Omega$ and $z,z_{1},z_{2}\in \mathbb{R}^{n}$
\begin{flalign}\label{assf0}
\begin{cases}
\ \nu\snr{z}^{p}\le F(x,z)\le L(1+\snr{z}^{2})^{\frac{q}{2}} \ \ \mbox{for all} \ \ (x,z)\in \Omega\times \mathbb{R}^{n} \\
\ z\mapsto F(\cdot,z) \ \ \mbox{is convex},
\end{cases}
\end{flalign}
where $0<\nu\le L$ are absolute constants and the exponents $(p,q)$ are so that
\begin{flalign}\label{pq0}
1<p< q\quad \mbox{and}\quad 0< q-p<\frac{p}{n-1}.
\end{flalign}
Let us consider also two measurable functions: $\psi\colon \Omega\to \mathbb{R}$ so that
\begin{flalign}\label{obs0}
\psi\in W^{1,q}(\Omega)
\end{flalign}
and 
\begin{flalign}\label{op3}
g\in W^{1,p}(\bar{\Omega}).
\end{flalign}
We are interested in some regularity properties of solutions of the obstacle problem
\begin{flalign}\label{op}
\mathcal{K}_{\psi,g}(\Omega)\ni w\mapsto \min\mathcal{F}(w,\Omega),
\end{flalign}
where 
\begin{flalign}\label{op1}
\mathcal{F}(w,\Omega):=\int_{\Omega}F(x,Dw) \ \dx
\end{flalign}
and
\begin{flalign}\label{op2}
\mathcal{K}_{\psi,g}(\Omega):=\left\{w\in W^{1,p}(\Omega)\colon w(x)\ge \psi(x) \ \mbox{a.e. in} \ \Omega \ \mbox{and} \ \left.w\right|_{\partial \Omega}=\left.g\right|_{\partial \Omega}\right\}.
\end{flalign}
In the following, we shall always assume that
\begin{flalign}\label{104}
\mathcal{K}_{\psi,g}(\Omega) \quad \mbox{is non-empty}.
\end{flalign}
Notice that if $v\in\mathcal{K}_{\psi,g}(\Omega)$ is a solution of problem \eqref{op}, then it is a local minimizer of the variational integral in \eqref{op1} with the obstacle constraint, in the sense of the following definition.
\begin{definition}\label{dm}
By local minimizer of \eqref{op1} with obstacle constraint we mean a map $v\in W^{1,p}(\Omega)$ such that
\begin{flalign*}
F(\cdot,Dv)\in L^{1}(\Omega), \qquad v(x)\ge \psi(x) \ \mbox{a.e. in} \ \Omega
\end{flalign*}
and whenever $\tilde{\Omega}\subseteq \Omega$ is an open set there holds that
\begin{flalign*}
\int_{\tilde{\Omega}}F(x,Dv) \ \dx \le \int_{\tilde{\Omega}}F(x,Dw) \ \dx \quad \mbox{for all} \ \ w\in v+W^{1,p}_{0}(\tilde{\Omega})\ \mbox{such that} \ w\ge \psi \ \mbox{a.e. in} \ \tilde{\Omega}.
\end{flalign*}
\end{definition}
In fact, if $\tilde{\Omega}\Subset \Omega$ is any open subset and $w\in v+W^{1,p}_{0}(\tilde{\Omega})$ is such that $w(x)\ge \psi(x)$ for a.e. $x\in \Omega$, then the map
\begin{flalign*}
\tilde{w}(x):=\begin{cases} 
\ w(x)\quad &\mbox{if} \ x\in \tilde{\Omega}\\
\ v(x)\quad &\mbox{if} \ x\in \Omega\setminus \tilde{\Omega}
\end{cases}
\end{flalign*}
belongs to $W^{1,p}(\Omega)$ since $v-w\in W^{1,p}_{0}(\tilde{\Omega})$ and by construction, $\tilde{w}\ge \psi$ a.e. in $\Omega$. Thus $\tilde{w}\in \mathcal{K}_{\psi,g}(\Omega)$ and
\begin{flalign*}
\int_{\tilde{\Omega}}F(x,Dv) \ \dx =&\int_{\Omega}F(x,Dv) \ \dx-\int_{\Omega\setminus \tilde{\Omega}}F(x,Dv) \ \dx \nonumber \\
\le &\int_{\Omega}F(x,D\tilde{w}) \ \dx -\int_{\Omega\setminus \tilde{\Omega}}F(x,Dv) \ \dx =\int_{\tilde{\Omega}}F(x,Dw) \ \dx.
\end{flalign*}
In particular, this argument shows that if $v\in \mathcal{K}_{\psi,g}(\Omega)$ is a solution of problem \eqref{op} and $\tilde{\Omega}\Subset \Omega$ is any open subset with boundary regular enough to allow for the concept of traces, then $v$ is a solution of the obstacle problem
\begin{flalign*}
\mathcal{K}_{\psi,v}(\tilde{\Omega})\ni w\mapsto \min \mathcal{F}(w,\tilde{\Omega}),
\end{flalign*}
where $\mathcal{K}_{\psi,v}(\tilde{\Omega})$ is defined as in \eqref{op2} with $g$ replaced by $v$, $\tilde{\Omega}$ instead of $\Omega$ and it is obviously non-empty, since $v\in \mathcal{K}_{\psi,v}(\tilde{\Omega})$.
\begin{remark}\label{bdd}
\emph{Being the outcomes of Theorems \ref{t1}-\ref{t2} local in nature, we do not assume more than \eqref{op3} for the regularity of the boundary datum $g$. Anyway, by $\eqref{pq0}_{2}$ and \cite[Lemma 2.1]{acbofo}, hypotheses \eqref{op3} makes problem \eqref{op} well posed.}
\end{remark}
\subsection{Notation} 
In this paper we denote by $c$ a general constant larger than one. Different occurences from line to line will be still denoted by $c$, while special occurrences will be denoted by $c_{1}, c_{2}, \tilde{c}$ and so on. Relevant dependencies on parameters will be emphasised using parentheses, i.e., $c_{1}= c_{1}(n,p)$ means that $c_{1}$ depends on $n,p$. In a similar fashion, by $o(\mathcal{l})$ we denote a quantity depending on the parameter $\mathcal{l}$ such that $o(\mathcal{l})\to 0$ when $\mathcal{l}$ goes to a relevant limit (typically $\mathcal{l} \to 0$ or $\mathcal{l} \to \infty$); also in this case the expression of $o(\mathcal{l})$ might vary from line to line and relevant dependencies are emphasized. We denote by $ B_{\rr}(x_{0}):=\{x \in \mathbb{R}^{n}\colon \snr{x-x_{0}}<\rr\}$ the open ball with center $x_{0}\in \mathbb{R}^{n}$ and radius $\rr>0$; when no ambiguity arises, we omit denoting the center as follows: $B_{\rr} \equiv B_{\rr}(x_{0})$. Very often, when not otherwise stated, different balls in the same context will share the same center. When considering function spaces of vector valued maps, such as $L^{p}(\Omega,\mathbb{R}^k)$, $W^{1,p}(\Omega,\mathbb{R}^{k})$ etc, we often abbreviate as $L^{p}(\Omega)$, $W^{1,p}(\Omega)$ and so on; the meaning will be clear from the context. Given any differentiable map $G: \Omega\times \mathbb{R}^{n}\to \mathbb{R}$, with $\partial_{z}G(x,z)$ we mean the derivative of $G$ with respect to the $z$ variable and by $\partial_{x}G(x,z)$ the derivative of $G$ in the $x$-variable, while, by $\partial^{2}_{z}G(x,z)$ we denote the second derivative in $z$ of $G$ and by $\partial^{2}_{x,z}G(x,z)$ the mixed one. For the sake of clarity, we shall adopt the shorthand notation
\begin{flalign*}
\texttt{data}_{\texttt{q}}:=\left(n,\nu,L,p,q,\alpha,\nr{\psi}_{W^{1+\alpha,q}(\Omega)}\right),\quad \texttt{data}_{\infty}:=\left(n,\nu,L,p,q,s,\mathcal{A}_{h,\psi},\nr{\psi}_{W^{2,\infty}(\Omega)}\right),
\end{flalign*}
see Sections \ref{pt1}-\ref{pt2} for more details on all the quantities involved.
\subsection{Auxiliary results} 
We start with some elementary facts on Sobolev functions. For a map $f \colon \Omega \to \mathbb{R}^{k}$, $k\ge 1$ and a vector $h \in \mathbb{R}^n$, we denote by $\tau_{h}\colon L^1(\Omega,\mathbb{\er}^{k}) \to L^{1}(\Omega_{|h|},\mathbb{R}^{k})$ the standard finite difference operator pointwise defined as
\begin{flalign*}
\tau_{h}f(x):=f(x+h)-f(x) \quad \mbox{for a.e.} \ x\in \Omega_{\snr{h}},
\end{flalign*}
where $\Omega_{|h|}:=\{x \in \Omega \, : \, 
\dist(x, \partial \Omega) > |h|\}$. It is clear that the finite difference operator is strictly connected with the weak differentiability of a function.
\begin{lemma}\label{l3} Let $B_{\rr}\subset B_{r}\Subset \Omega$ be two balls, $h\in \mathbb{R}^{n}$ be a vector with $\snr{h}<\frac{1}{4}\min\{r-\rr,\dist(\partial B_{r},\partial \Omega)\}$ and $f\in W^{1,t}(\Omega,\mathbb{R}^{k})$ for some $t\in [1,\infty)$. Then
\begin{flalign*}
\int_{B_{\rr}}\snr{\tau_{h}f}^{t} \ \dx\le \snr{h}^{t}\int_{B_{r}}\snr{Df}^{t} \ \dx.
\end{flalign*}
\end{lemma}
Controlling a suitable Lebesgue norm of the finite difference of a function implies weak differentiability.
\begin{lemma}\label{l8}
Let $B_{\rr}\subset B_{r}\Subset \Omega$ be two balls. If $f\in L^{t}(\Omega,\mathbb{R}^{k})$, $t\in (1,\infty)$, is a map such that
\begin{flalign*}
\int_{B_{\rr}}\snr{\tau_{h}f}^{t} \ \dx \le S^{t}\snr{h}^{t}
\end{flalign*}
for all vectors $h\in \mathbb{R}^{n}$ with $\snr{h}<\frac{1}{4}\min\{r-\rr,\dist(\partial B_{r},\partial \Omega)\}$, then
\begin{flalign*}
f\in W^{1,t}(B_{\rr},\mathbb{R}^{k})\quad \mbox{and}\quad \nr{Df}_{L^{t}(B_{\rr})}\le S.
\end{flalign*}
\end{lemma}
The next result explains how to control translations.
\begin{lemma}\label{l2}
Let $B_{\rr}\subset B_{r}\Subset \Omega$ be two balls, $h\in \mathbb{R}^{n}$ be a vector so that $\snr{h}<\frac{1}{4}\min\{r-\rr,\dist(\partial B_{r},\partial \Omega)\}$ and $f\in L^{t}(\Omega,\mathbb{R}^{k})$ for some $t\in [1,\infty)$. Then
\begin{flalign*}
\int_{B_{\rr}}\left[\snr{f(x)}^{2}+\snr{f(x+h)}^{2}\right]^{\frac{t}{2}} \ \dx \le c(n,t)\int_{B_{r}}\snr{f(x)}^{t} \ \dx.
\end{flalign*}
\end{lemma}
We now recall a few basic facts concerning fractional Sobolev spaces. \begin{definition}\label{fra1def}
Let $\alpha \in (0,1)$, $p \in [1, \infty)$, $k \in \mathbb{N}$, and let $\Omega \subset \mathbb{R}^n$ be an open subset with $n\ge 2$ (we allow for the case $\Omega =\mathbb{R}^{n}$). The fractional Sobolev space $W^{\alpha ,p}(\Omega,\mathbb{R}^k )$ is defined prescribing that $f \colon \Omega \to \mathbb{R}^k$ belongs to  $W^{\alpha ,p}(\Omega,\mathbb{R}^k )\equiv W^{\alpha ,p}(\Omega)$ iff the following Gagliardo type norm is finite:
\begin{flalign*}
\nr{f}_{W^{\alpha,p}(\Omega)}&:=\nr{f}_{L^{p}(\Omega)}+\left(\int_{\Omega} \int_{\Omega}  
\frac{\snr{f(x)
- f(y) }^{p}}{\snr{x-y}^{n+\alpha p}} \ \dx \, \dy \right)^{1/p}=:\nr{f}_{L^{p}(\Omega)}+[f]_{\alpha,p;\Omega}.
\end{flalign*}
Accordingly, in the case $\alpha = [\alpha]+\{\alpha\}\in \en + (0,1)>1$, we say that $f\in W^{\alpha ,p}(\Omega,\mathbb{R}^k )$ iff the following quantity is finite
\begin{flalign*}
\nr{f}_{W^{\alpha,p}(\Omega)}:=\nr{f}_{W^{[\alpha],p}(\Omega)}+[D^{[\alpha]}f]_{\{\alpha\},p;\Omega}.
\end{flalign*}
The local variant $W^{\alpha ,p}_{loc}(\Omega,\er^k )$ is defined by requiring that $f \in W^{\alpha ,p}_{loc}(\Omega,\mathbb{R}^k )$ iff $f \in W^{\alpha ,p}(\tilde{\Omega},\mathbb{R}^k)$ for every open subset $\tilde{\Omega} \Subset \Omega$. 
\end{definition}

\begin{definition}\label{fra2def}
Let $\alpha \in (0,1)$, $p\in [1, \infty)$, $k \in \mathbb{R}^{n}$, and let $\Omega \subset \mathbb{R}^n$ be an open subset with $n\geq 2$. The Nikol'skii space $N^{\alpha,p}(\Omega,\mathbb{R}^k )$ is defined prescribing that $f \in N^{\alpha,p}(\Omega,\mathbb{R}^k )$ iff 
$$\nr{f}_{N^{\alpha,p}(\Omega )} :=\nr{f}_{L^{p}(\Omega)} + \left(\sup_{\snr{h}\not=0}\, \int_{\Omega_{|h|}} 
\frac{\snr{f(x+h)
- f(x) }^{p}}{|h|^{\alpha p}} \ dx  \right)^{1/p}\;.$$
The local variant $N^{\alpha,p}_{\loc}(\Omega,\mathbb{R}^{k} )$ is defined by requiring that $f \in N^{\alpha,p}_{loc}(\Omega,\mathbb{R}^{k} )$ iff $f \in N^{\alpha,p}(\tilde \Omega,\mathbb{R}^{k})$ for every open subset $\tilde{\Omega} \Subset \Omega$.
\end{definition}
Moreover we have that 
\begin{flalign}\label{33}
W^{\alpha ,p}(\Omega,\er^k)\subsetneqq N^{\alpha,p}(\Omega,\er^k)\subsetneqq
W^{\beta,p}(\Omega,\er^k)\quad \mbox{for every} \ \ \beta<\alpha,
\end{flalign}
holds for sufficiently regular domains $\Omega$. Notice that, given any ball $B_{\rr}\Subset \Omega$ such that $\dist(\partial B_{\rr},\partial{\Omega})>0$, a function $f\in N^{\alpha,q}(\Omega,\mathbb{R}^{k})$ and a vector $h\in \mathbb{R}^{n}$ with $\snr{h}<\frac{1}{4}\dist(\partial B_{\rr},\partial{\Omega})$, than Definition \ref{fra2def} and $\eqref{33}$ immediately imply that
\begin{flalign}\label{0}
\left(\int_{B_{\rr}}\snr{f(x+h)-f(x)}^{p} \ \dx\right)^{\frac{1}{p}}\le& \snr{h}^{\alpha}\left(\sup_{\snr{h}\not =0}\int_{\Omega_{\snr{h}}}\frac{\snr{f(x+h)-f(x)}^{p}}{\snr{h}^{\alpha p}} \ \dx\right)^{\frac{1}{p}}\nonumber \\
\le& \snr{h}^{\alpha}\nr{f}_{N^{\alpha,p}(\Omega)}\le c(n,p)\snr{h}^{\alpha}\nr{f}_{W^{\alpha,p}(\Omega)}.
\end{flalign}
A local, quantified version of \eqref{33} in the next lemma. 
\begin{lemma}\label{l4}\emph{\cite{avkumi}}
Let $B_{r}\Subset \mathbb{R}^{n}$ be a ball with $r\le 1$, $f\in L^{p}(B_{r},\mathbb{R}^{k})$, $p>1$ and assume that, for $\alpha \in (0,1]$, $S\ge 1$ and concentric balls $B_{\rr}\Subset B_{r}$, there holds
\begin{flalign*}
\nr{\tau_{h}f}_{L^{p}(B_{\rr})}\le S\snr{h}^{\alpha}\quad \mbox{for every} \ h\in \mathbb{R}^{n} \ \mbox{with} \ 0<\snr{h}\le \frac{r-\rr}{K}, \ \mbox{where} \ K\ge 1.
\end{flalign*}
Then $f\in W^{\beta,p}(B_{\rr},\mathbb{R}^{k})$ whenever $\beta\in (0,\alpha )$ and
\begin{flalign*}
\nr{f}_{W^{\beta,p}(B_{\rr})}\le\frac{c}{(\alpha -\beta)^{1/p}}
\left(\frac{r-\rr}{K}\right)^{\alpha -\beta}S+c\left(\frac{K}{r-\rr}\right)^{n/p+\beta} \nr{f}_{L^{p}(B_{r})}
\end{flalign*}
holds, where $c= c(n,p)$. 
\end{lemma}
The next is the embedding theorem for fractional Sobolev spaces.
\begin{lemma}\label{l9}\emph{\cite{pala}}
Let $f\in W^{\alpha,p}(\Omega,\mathbb{R}^{k})$, with $p\ge 1$, $\alpha\in (0,1]$ such that $\alpha p<n$ and let $\Omega\subset \mathbb{R}^{n}$ be a bounded, Lipschitz domain. Then 
\begin{itemize}
    \item[-]$\alpha p<n\Rightarrow f\in L^{\frac{n p}{n-\alpha p}}(\Omega,\mathbb{R}^{k})$ with $\nr{f}_{L^{\frac{np}{n-\alpha p}}(\Omega)}\le c\nr{f}_{W^{\alpha,p}(\Omega)}$;
    \item[-]$\alpha p=n\Rightarrow f\in L^{t}(\Omega,\mathbb{R}^{k})$ for all $t\in [p,\infty)$, with $\nr{f}_{L^{t}(\Omega)}\le c\nr{f}_{W^{\alpha,p}(\Omega)}$;
    \item[-]$\alpha p>n\Rightarrow f\in C^{0,\frac{\alpha p-n}{p}}(\Omega,\mathbb{R}^{k})$ with $\nr{f}_{0,\frac{\alpha p-n}{p};\Omega}\le c\nr{f}_{W^{\alpha,p}(\Omega)}$;
\end{itemize}
with $c$ depending at the most from $(n,\alpha,p,t,[\partial \Omega]_{0,1},\diam(\Omega))$.
\end{lemma}
We refer to \cite{pala} for a survey on this matter. We close this section by reporting some informations on well-known tools in the Calculus of Variations. For constant $\tilde{c}\in [0,1]$ and $z\in \mathbb{R}^{n}$ we introduce the auxiliary vector field
\begin{flalign*}
V_{\tilde{c},t}(z):=(\tilde{c}^{2}+\snr{z}^{2})^{\frac{t-2}{4}}z\qquad t\in \{p,q\},
\end{flalign*}
which turns out to be very convenient in handling the monotonicity properties of certain operators.
\begin{lemma}\label{l1}\emph{\cite{ha}}
For any given $z_{1},z_{2}\in \mathbb{R}^{n}$, $z_{1}\not=z_{2}$ there holds that
\begin{flalign*}
\snr{V_{\tilde{c},t}(z_{1})-V_{\tilde{c},t}(z_{2})}^{2}\sim (\tilde{c}^{2}+\snr{z_{1}}^{2}+\snr{z_{2}}^{2})^{\frac{t-2}{2}}\snr{z_{1}-z_{2}}^{2},
\end{flalign*}
where the constants implicit in "$\sim$" depend only from $(n,t)$.
\end{lemma}
Another useful result is the following
\begin{lemma}\label{l6}\emph{\cite{acfu}}
Let $t>-1$, $\tilde{c}\in [0,1]$ and $z_{1},z_{2}\in \mathbb{R}^{n}$ be so that $\tilde{c}+\snr{z_{1}}+\snr{z}_{2}>0$. Then
\begin{flalign*}
\int_{0}^{1}\left[\tilde{c}^{2}+\snr{z_{1}+\lambda(z_{2}-z_{1})}^{2}\right]^{\frac{t}{2}} \ \d\lambda\sim (\tilde{c}^{2}+\snr{z_{1}}^{2}+\snr{z_{2}}^{2})^{\frac{t}{2}},
\end{flalign*}
with constants implicit in "$\sim$" depending only from $t$.
\end{lemma}
Finally, the iteration lemma.
\begin{lemma}\label{l5}
Let $h\colon [\rr_{0},\rr_{1}]\to \mathbb{R}$ be a non-negative and bounded function, and let $\theta \in (0,1)$, $A,B,\gamma_{1},\gamma_{2}\ge 0$ be numbers. Assume that
\begin{flalign*}
h(t)\le \theta h(s)+\frac{A}{(s-t)^{\gamma_{1}}}+\frac{B}{(s-t)^{\gamma_{2}}}
\end{flalign*}
holds for all $\rr_{0}\le t<s\le \rr_{1}$. Then the following inequality holds
\begin{flalign*}
h(\rr_{0})\le c(\theta,\gamma_{1},\gamma_{2})\left\{\frac{A}{(\rr_{1}-\rr_{0})^{\gamma_{1}}}+\frac{B}{(\rr_{1}-\rr_{0})^{\gamma_{2}}}\right\}.
\end{flalign*}
\end{lemma}

\subsection{Existence and uniqueness}\label{ex}
The existence of a solution of problem \eqref{op} easily follows from direct methods, we briefly report a sketch for completeness. Let $\{v_{j}\}_{j \in \N}\subset \mathcal{K}_{\psi,g}(\Omega)$ be a minimizing sequence. Therefore,
\begin{flalign}\label{3}
\int_{\Omega}F(x,Dv_{j}) \ \dx \to_{j\to \infty}m:=\inf_{w\in \mathcal{K}_{\psi,g}(\Omega)}\mathcal{F}(w,\Omega).
\end{flalign}
This means that, for $j\in \N$ sufficiently large there holds that
\begin{flalign}\label{1}
\int_{\Omega}F(x,Dv_{j}) \ \dx \le m+1.
\end{flalign}
Combining $\eqref{assf}_{1}$, \eqref{1} and Poincar\'e inequality we directly have
\begin{flalign*}
\nr{Dv_{j}}_{W^{1,p}(\Omega)}^{p}\le \frac{m+1}{\nu}\quad \mbox{and}\quad \nr{v_{j}}_{W^{1,p}(\Omega)}\le c(n,p)\left[\nr{g}_{W^{1,p}(\Omega)}+\left(\frac{m+1}{\nu}\right)^{\frac{1}{p}}\right],
\end{flalign*}
thus, up to extract a (non-relabelled) subsequence, we get
\begin{flalign}\label{2}
v_{j}\rightharpoonup_{j\to \infty}v \ \ \mbox{in} \ \ W^{1,p}(\Omega) \quad \mbox{and} \quad v_{j}\to_{j\to \infty}v \ \ \mbox{in} \ \ L^{p}(\Omega).
\end{flalign}
By $\eqref{2}$ we have that $v(x)\ge \psi(x)$ a.e. in $\Omega$ and $\left.v\right|_{\partial \Omega}=\left.g\right|_{\partial \Omega}$, thus $v\in \mathcal{K}_{\psi,g}(\Omega)$. Using $\eqref{assf0}_{2}$, $\eqref{2}_{1}$, weak lower semicontinuity and \eqref{3} we can conclude that
\begin{flalign*}
\int_{\Omega}F(x,Dv) \ \dx \le \liminf_{j\to \infty}\int_{\Omega}F(x,Dv_{j}) \ \dx =m,
\end{flalign*}
so $v\in \mathcal{K}_{\psi,g}(\Omega)$ solves \eqref{op}. In case we ask for strict convexity rather than just $\eqref{assf0}_{2}$, we can guarantee that $v\in \mathcal{K}_{\psi,g}(\Omega)$ is actually the unique solution of our problem: in fact, if $v,\tilde{v}\in \mathcal{K}_{\psi,g}(\Omega)$ are both solutions of problem \eqref{op}, we can define $\bar{v}:=\frac{\tilde{v}+v}{2}$ and get
\begin{flalign*}
\int_{\Omega}F(x,D\bar{v}) \ \dx <\frac{1}{2}\int_{\Omega}F(x,D\tilde{v}) \ \dx+\frac{1}{2}\int_{\Omega}F(x,Dv) \ \dx=m,
\end{flalign*}
which is clearly a nonsense, since $\bar{v}\in \mathcal{K}_{\psi,g}(\Omega)$. 
\section{Relaxation}\label{re} In this section we shall provide a meaningful definition of relaxation for problem \eqref{op} in the spirit of \cite{acbofo,eslemi,ma2}. Given the local nature of our main theorems, in the following we will not consider boundary conditions. Let $\tilde{\Omega}\Subset \Omega$ be an open subset and define
\begin{flalign*}
&\mathcal{K}_{\psi}:=W^{1,p}_{loc}(\Omega)\cap\left\{w\in W^{1,p}(\tilde{\Omega})\colon w(x)\ge \psi(x) \ \mbox{for a.e.} \ x\in \tilde{\Omega}\right\},\\
&\mathcal{K}^{*}_{\psi}:=W^{1,p}_{loc}(\Omega)\cap\left\{w\in W^{1,q}(\tilde{\Omega})\colon w(x)\ge \psi(x) \ \mbox{for a.e.} \ x\in\tilde{\Omega}\right\}.
\end{flalign*}
Being convex and closed, $\mathcal{K}_{\psi}$ is a Banach subspace of $W^{1,p}(\tilde{\Omega})$ and $\mathcal{K}_{\psi}^{*}$ is a Banach subspace of $W^{1,q}(\tilde{\Omega})$. 
\begin{lemma}\label{l10}
Class $\mathcal{K}_{\psi}^{*}$ is dense in $\mathcal{K}_{\psi}$ with respect to the $W^{1,p}$-norm.
\end{lemma}
\begin{proof}
Let $\{\phi_{j}\}_{j \in \N}$ be a family of standard, non-negative, radially symmetric mollifiers so that
\begin{flalign}\label{molli}
\phi \in C^{\infty}_{c}(B_{1}),\qquad \nr{\phi}_{L^{1}(\mathbb{R}^{n})}=1, \qquad \phi_{j}(x):=j^{n}\phi\left(jx\right)
\end{flalign}
and set $\tilde{\psi}_{j}:=\psi*\phi_{j}$ and $\tilde{w}_{j}:=w*\phi_{j}$, where $w\in \mathcal{K}_{\psi}$. By the properties of convolution and \eqref{obs0} we have that
\begin{flalign}\label{21}
\begin{cases}
\ \{\tilde{\psi}_{j}\}_{j\in \N}\subset C^{\infty}_{loc}(\Omega)\quad \mbox{and}\quad \tilde{\psi}_{j}\to_{j\to \infty}\psi \ \ \mbox{in} \ \ W^{1,q}(\tilde{\Omega})\\
\ \{\tilde{w}_{j}\}_{j\in \N}\subset C^{\infty}_{loc}(\Omega)\quad \mbox{and}\quad \tilde{w}_{j}\to_{j\to \infty}w \ \ \mbox{in} \ \ W^{1,p}(\tilde{\Omega}).
\end{cases}
\end{flalign}
Furthermore, there holds
\begin{flalign}\label{22}
\tilde{w}_{j}(x)=\int_{B_{1}}\phi(y)w(x+j^{-1}y) \ \dx \ge \int_{B_{1}}\phi(y)\psi(x+j^{-1}y) \ \dx=\tilde{\psi}_{j}(x) \quad \mbox{for all} \ \ x\in \tilde{\Omega}.
\end{flalign}
Now set $w_{j}:=\tilde{w}_{j}-\tilde{\psi}_{j}+\psi$. From $\eqref{pq0}_{1}$, \eqref{21}, \eqref{obs0} and \eqref{22}, it directly follows that $\{w_{j}\}_{j\in \N}\subset \mathcal{K}^{*}_{\psi}$ and
\begin{flalign*}
\nr{w_{j}-w}_{L^{p}(\tilde{\Omega})}\le \left[\nr{\tilde{w}_{j}-w}_{L^{p}(\tilde{\Omega})}+\nr{\tilde{\psi}_{j}-\psi}_{L^{p}(\tilde{\Omega})}\right]\to_{j\to \infty}0.
\end{flalign*}
\end{proof}
Once established this density result, we can consider the relaxed functional
\begin{flalign*}
\mathcal{K}_{\psi}\ni w\mapsto\bar{\mathcal{F}}(w,\tilde{\Omega}):=\inf_{\mathcal{C}(w)}\left\{\liminf_{j\to \infty}\int_{\tilde{\Omega}}F(x,Dw_{j}) \ \dx\right\},
\end{flalign*}
where
\begin{flalign*}
&\mathcal{C}(w):=\left\{\{w_{j}\}_{j \in \N}\subset \mathcal{K}^{*}_{\psi}\colon w_{j}\rightharpoonup_{j\to \infty}w \ \ \mbox{in} \ \ W^{1,p}_{loc}(\Omega)\right\},
\end{flalign*}
Notice that $\mathcal{C}(w)$ is non-empty, given that the sequence $\{\tilde{w}_{j}-\tilde{\psi}_{j}+\psi\}_{j\in \N}$, where $\{\tilde{w}_{j}\}_{j\in \N}$ and $\{\tilde{\psi}_{j}\}_{j \in \N}$ are as in \eqref{21}, belongs to $\mathcal{C}(w)$, (recall \eqref{22}). Let us connect functional $\bar{\mathcal{F}}$ with the original one appearing in problem \eqref{op}. By $\eqref{assf0}_{2}$ and weak-lower semicontinuity, we have
\begin{flalign}\label{27}
\bar{\mathcal{F}}(w,\tilde{\Omega})\ge \mathcal{F}(w,\tilde{\Omega}) \quad \mbox{for all} \ \ w\in \mathcal{K}_{\psi}.
\end{flalign}
Moreover, if $w\in \mathcal{K}_{\psi}^{*}$, we get in addition that the regularized sequence in $\eqref{21}_{2}$, $\{\tilde{w}_{j}\}_{j\in \N}$ strongly converges to $w$ in $W^{1,q}(\tilde{\Omega})$, therefore, using a well-known variant of Lebesgue dominated convergence theorem, we end up with
\begin{flalign}\label{28}
\liminf_{j\to \infty}\int_{\tilde{\Omega}}F(x,D\tilde{w}_{j}) \ \dx=\lim_{j\to \infty}\int_{\tilde{\Omega}}F(x,D\tilde{w}_{j}) \ \dx =\int_{\tilde{\Omega}}F(x,Dw) \ \dx.
\end{flalign}
From \eqref{27} and \eqref{28} we can conclude that if $w\in \mathcal{K}_{\psi}^{*}$, then $\bar{\mathcal{F}}(w,\tilde{\Omega})=\mathcal{F}(w,\tilde{\Omega})$. As in \cite{eslemi}, we then define the gap functional
\begin{flalign*}
\mathcal{L}(w,\tilde{\Omega}):=\begin{cases} \ \bar{\mathcal{F}}(w,\tilde{\Omega})-\mathcal{F}(w,\tilde{\Omega}) \ \ &\mbox{if} \ \ \mathcal{F}(w,\tilde{\Omega})<\infty\\
\ 0 \ \ &\mbox{if} \ \ \mathcal{F}(w,\tilde{\Omega})=\infty.
\end{cases}
\end{flalign*}
If $w\in \mathcal{K}_{\psi}$ is so that $\mathcal{L}(w,\tilde{\Omega})=0$, then there exists a sequence $\{w_{j}\}_{j\in \N}\subset W^{1,p}_{loc}(\Omega)\cap W^{1,q}(\tilde{\Omega})$ so that
\begin{flalign}\label{31}
\begin{cases}
\ w_{j}\rightharpoonup_{j\to \infty}w \ \ \mbox{in} \ \ W^{1,p}_{loc}(\Omega)\\
\ w_{j}\ge \psi \ \ \mbox{a.e. in} \ \tilde{\Omega}\\
\ \int_{\tilde{\Omega}}F(x,Dw_{j}) \ \dx \to_{j\to \infty}\int_{\tilde{\Omega}}F(x,Dw) \ \dx,
\end{cases}
\end{flalign}
see \cite[Section 4]{acbofo}. This is actually the key to show that the vanishing of the Lavrentiev gap functional assures that
\begin{flalign*}
\inf_{w\in \mathcal{K}_{\psi}}\int_{\tilde{\Omega}}F(x,Dw) \ \dx\stackrel{\eqref{assf0}_{1}}{=}\min_{w\in \mathcal{K}_{\psi}}\int_{\tilde{\Omega}}F(x,Dw) \ \dx=\inf_{w\in \mathcal{K}^{*}_{\psi}}\int_{\tilde{\Omega}}F(x,Dw) \ \dx.
\end{flalign*}
Indeed, since $\mathcal{K}^{*}_{\psi}\subset \mathcal{K}_{\psi}$, we have
\begin{flalign}\label{30}
\inf_{w\in \mathcal{K}^{*}_{\psi}}\int_{\tilde{\Omega}}F(x,Dw) \ \dx\ge \inf_{w\in \mathcal{K}_{\psi}}\int_{\tilde{\Omega}}F(x,Dw) \ \dx
\end{flalign}
and if we assume that $\mathcal{L}(v,\tilde{\Omega})=0$, where $v\in \mathcal{K}_{\psi}$ is so that
\begin{flalign*}
\int_{\tilde{\Omega}}F(x,Dv) \ \dx =\min_{w\in \mathcal{K}_{\psi}}\int_{\tilde{\Omega}}F(x,Dw) \ \dx,
\end{flalign*}
then we can find a sequence $\{v_{j}\}_{j\in \N}\subset W^{1,p}_{loc}(\Omega)\cap W^{1,q}(\tilde{\Omega})$ as in \eqref{31} which realizes \eqref{30}.

\begin{remark}
\emph{We saw before that for any given map $w\in \mathcal{K}_{\psi,g}(\Omega)$, condition
\begin{flalign}\label{nolavq}
\mathcal{L}(w,\tilde{\Omega})=0\quad \mbox{for all open subsets} \ \tilde{\Omega}\Subset \Omega
\end{flalign}
yields \eqref{31}, which is a crucial tool in the proof of Theorems \ref{t1}-\ref{t2}. In particular, if we do not assume any specific underlying structure for the integrand $F$, \eqref{nolavq} needs to be taken as an assumption. On the other hand, by \cite[Section 5]{eslemi} and \cite[Section 3.5]{haha}, under suitable assumptions, we know that there are several models, such as}
\begin{flalign*}
&\mathcal{F}_{1}(w,\Omega):=\int_{\Omega}\snr{Dw}^{p(x)} \ \dx\ \  &&1<\inf_{x\in \Omega}p(x)\le p(\cdot)\le \sup_{x\in \Omega}p(x)<\infty\\
&\mathcal{F}_{2}(w,\Omega):=\int_{\Omega}\left[\snr{Dw}^{p}+a(x)\snr{Dw}^{q}\right] \ \dx\ \ &&0\le a(\cdot)\in C^{0,\alpha}(\Omega)\\ &\mathcal{F}_{3}(w,\Omega):=\int_{\Omega}\varphi(x,\snr{Dw}) \ \dx, \ \ &&\mbox{see assumptions of \cite[Section 3]{haok}},
\end{flalign*}
\emph{just to quote the most popular, realizing \eqref{nolavq}. In fact, whenever $w\in W^{1,1}(\Omega)$ is so that $\mathcal{F}_{i}(w,\Omega)<\infty$ and $\tilde{\Omega}\Subset \Omega$ is an open subset, then we can regularize $w$ via a family of mollifiers as in \eqref{molli}, thus obtaining a sequence $\{\bar{w}_{j}\}_{j \in \N}\in C^{\infty}_{loc}(\Omega)$ satisfying}
\begin{flalign*}
\bar{w}_{j}\to_{j\to \infty}w \ \ \mbox{in} \ \ W^{1,p}(\tilde{\Omega})\quad \mbox{\emph{and}}\quad \mathcal{F}_{i}(\bar{w}_{j},\tilde{\Omega})\to_{j\to \infty}\mathcal{F}_{i}(w,\tilde{\Omega}),
\end{flalign*}
\emph{for all $i\in \{1,2,3\}$. We can then apply the trick presented in the proof of Lemma \ref{l10} and make minor changes to the techniques in \cite[Section 5]{eslemi} and \cite[Section 3.5]{haha} to build a sequence $\{\tilde{v}_{j}\}_{j \in \N}\subset W^{1,p}_{loc}(\Omega)\cap W^{1,q}(\tilde{\Omega})$ matching \eqref{31}. Given that \eqref{31} and \eqref{nolavq} are equivalent, under the appropriate set of assumptions on exponents or coefficients, our results cover models $\mathcal{F}_{1}$-$\mathcal{F}_{3}$, see Sections \ref{pt1}-\ref{pt2} for more details.}

\end{remark} 
\section{Proof of Theorem \ref{t1}}\label{pt1} To prove Theorem \ref{t1}, we need to assume something more on both the integrand $F$ and on the obstacle $\psi$. Precisely, we ask that the Carath\'eodory integrand $F\colon \Omega\times \mathbb{R}^{n}\to \mathbb{R}$ verifies
\begin{flalign}\label{assf}
\begin{cases}
\ \nu \snr{z}^{p}\le F(x,z)\le L(1+\snr{z}^{2})^{\frac{q}{2}}\\
\ z\mapsto F(\cdot,z)\in C^{1}(\mathbb{R}^{n}) \\
\ \left[(\partial_{z} F(x,z_{1})-\partial_{z} F_{z}(x,z_{2}))\cdot (z_{1}-z_{2})\right]\ge \nu (\mu^{2}+\snr{z_{1}}^{2}+\snr{z_{2}}^{2})^{\frac{p-2}{2}}\snr{z_{1}-z_{2}}^{2}\\
\ \snr{\partial_{z} F(x_{1},z)-\partial_{z} F(x_{2},z)}\le L \snr{x-y}^{\alpha}(1+\snr{z}^{2})^{\frac{q-1}{2}}, \quad \alpha \in (0,1],
\end{cases}
\end{flalign}
for all $x,x_{1},x_{2}\in \Omega$ and all $z,z_{1},z_{2}\in \mathbb{R}^{n}$ with $0<\nu\le L$ and $\mu\in [0,1]$ absolute constants. The exponents $(p,q)$ are such that
\begin{flalign}\label{pq}
1<p\le q<p\left(1+\frac{\alpha}{n}\right)
\end{flalign}
and the obstacle $\psi\colon \Omega \to \mathbb{R}$ satisfies
\begin{flalign}\label{obbd}
\psi\in W^{1+\alpha,q}(\Omega).
\end{flalign}
Some comments are in order. First, notice that \eqref{pq0} holds also in this case. Moreover, $\eqref{assf}_{3}$ implies that
\begin{flalign}\label{sconv}
z\mapsto F(\cdot,z)\ \ \mbox{is strictly convex}
\end{flalign}
and, as a consequence of $\eqref{assf}_{1,2}$ and \eqref{sconv}, we get that
\begin{flalign}\label{23}
\snr{\partial_{z} F(x,z)}\le c(n,L,q)(1+\snr{z}^{2})^{\frac{q-1}{2}},
\end{flalign}
see \cite[Lemma 2.1]{ma1}. Furthermore, by Lemma \ref{l9}, 
\begin{flalign*}
\begin{cases}
\ D\psi\in L^{\frac{nq}{n-\alpha q}}(\Omega,\mathbb{R}^{n}) \quad &\mbox{if} \ \  \alpha q<n\\
\ D\psi \in L^{t}(\Omega,\mathbb{R}^{n}) \ \ \mbox{for all} \ t\ge q\quad &\mbox{if} \ \ \alpha q\ge n
\end{cases}
\end{flalign*}
so in any case $D\psi\in L^{q}(\Omega)$ and also \eqref{obs0} still holds true. This legalizes our final assumption: condition \eqref{nolavq} is verified by the solution $v\in \mathcal{K}_{\psi,g}(\Omega)$ of problem \eqref{op}, (recall the content of Section \ref{ex} and \eqref{sconv}).
Finally, by \eqref{obbd} and \eqref{0} we can conclude that if $B_{\rr}\Subset \Omega$ is a ball and $h\in \mathbb{R}^{n}$ is any vector with $\snr{h}<\frac{\dist(\partial B_{\rr},\partial \Omega)}{4}$, then 
\begin{flalign}\label{34}
\left(\int_{B_{\rr}}\snr{D\psi(x+h)-D\psi(x)}^{q} \ \dx\right)^{\frac{1}{q}}\le c\snr{h}^{\alpha}\nr{\psi}_{W^{1+\alpha,q}(\Omega)},
\end{flalign}
for $c=c(n,q,\alpha)$. For the ease of exposition, we shall split the proof into two moments: first we are going to show the higher integrability result and then derive extra fractional differentiability.
\subsection{Higher integrability}\label{hi}
Let $v\in \mathcal{K}_{\psi}(\Omega)$ be the solution of problem \eqref{op}. Let us fix a ball $B_{\rr}\Subset \Omega$ with $\rr\in (0,1]$. Since $v$ satisfies \eqref{nolavq}, by \eqref{31}, this means that there exists a sequence $\{\tilde{v}_{j}\}_{j \in \N}\subset W^{1,p}_{loc}(\Omega)\cap W^{1,q}(B_{\rr})$ such that
\begin{flalign}\label{35}
\begin{cases}
\ \tilde{v}_{j}\rightharpoonup_{j\to \infty}v \ \ \mbox{in} \ \ W^{1,p}_{loc}(\Omega)\\
\ \tilde{v}_{j}\ge \psi \ \ \mbox{a.e. in} \ B_{\rr}\\
\ \int_{B_{\rr}}F(x,D\tilde{v}_{j}) \ \dx \to_{j\to \infty}\int_{B_{\rr}}F(x,Dv) \ \dx.
\end{cases}
\end{flalign}
We introduce a suitable family of regularized problems. To do so, we set
\begin{flalign}\label{fj} F_{j}(x,z):=F(x,z)+\frac{\varepsilon_{j}}{q}(1+\snr{z}^{2})^{\frac{q}{2}} \ \ \mbox{for all} \ \ (x,z)\in B_{\rr}\times \mathbb{R}^{n}
\end{flalign}
and consider the obstacle problem
\begin{flalign}\label{opj}
\mathcal{K}_{\psi,\tilde{v}_{j}}^{*}(B_{\rr})\ni w\mapsto \min\int_{B_{\rr}}F_{j}(x,Dw) \ \dx,
\end{flalign}
where 
\begin{flalign}\label{eps}
\varepsilon_{j}:=\left[1+j+\nr{D\tilde{v}_{j}}_{L^{q}(B_{\rr})}^{q}\right]^{-1}
\end{flalign}
and
\begin{flalign}\label{kjd}
\mathcal{K}^{*}_{\psi,\tilde{v}_{j}}(B_{\rr}):=\left\{w\in W^{1,q}(B_{\rr})\colon w(x)\ge \psi(x) \ \mbox{a.e. in} \ B_{\rr} \ \mbox{and} \ \left.w\right|_{\partial{B_{\rr}}}=\left.\tilde{v}_{j}\right|_{\partial{B_{\rr}}}\right\}.
\end{flalign}
Notice that $\mathcal{K}_{\psi,\tilde{v}_{j}}^{*}(B_{\rr})\not =\emptyset$ since by \eqref{35}, $\tilde{v}_{j}\in \mathcal{K}_{\psi,\tilde{v}_{j}}^{*}(B_{\rr})$. Recalling assumptions \eqref{assf}, it is easy to see that the integrand in \eqref{fj} satisfies
\begin{flalign}\label{assfj}
\begin{cases}
\ \nu \snr{z}^{p}+\frac{\varepsilon_{j}}{q}(1+\snr{z}^{2})^{\frac{q}{2}}\le F_{j}(x,z)\le \left(L+\frac{\varepsilon_{j}}{q}\right)(1+\snr{z}^{2})^{\frac{q}{2}}\\
\ z\mapsto F_{j}(x,z)\in C^{1}(\mathbb{R}^{n})\\
\ \left[(\partial_{z} F_{j}(x,z_{1})-\partial_{z} F_{j}(x,z_{2}))\cdot (z_{1}-z_{2})\right]\ge \nu (\mu^{2}+\snr{z_{1}}^{2}+\snr{z_{2}}^{2})^{\frac{p-2}{2}}\snr{z_{1}-z_{2}}^{2}\\
\ \snr{\partial_{z} F_{j}(x_{1},z)-\partial_{z} F_{j}(x_{2},z)}\le L \snr{x-y}^{\alpha}(1+\snr{z}^{2})^{\frac{q-1}{2}},
\end{cases}
\end{flalign}
whenever $x,x_{1},x_{2}\in B_{\rr}$ and $z,z_{1},z_{2}\in \mathbb{R}^{n}$ for absolute constants $0<\nu\le L$ and $\mu\in [0,1]$. Notice that $\eqref{assfj}_{3}$ yields that $z\mapsto F_{j}(\cdot,z)$ is strictly convex so, again by $\eqref{assfj}_{1,2}$ it follows that
\begin{flalign}\label{4}
\snr{\partial_{z} F_{j}(x,z)}\le c(n,L,q)(1+\snr{z}^{2})^{\frac{q-1}{2}} \ \ \mbox{for all} \ \ (x,z)\in B_{\rr}\times \mathbb{R}^{n}.
\end{flalign}
Using the content of Section \ref{ex}, we see that there exists a unique solution $v_{j}\in \mathcal{K}_{\psi,\tilde{v}_{j}}^{*}(B_{\rr})$ of problem \eqref{opj} and the following variational inequality holds
\begin{flalign}\label{vej}
\int_{B_{\rr}}\partial_{z} F_{j}(x,Dv_{j})\cdot D(w-v_{j}) \ \dx\ge 0\quad \mbox{for all} \ \ w\in \mathcal{K}_{\psi,\tilde{v}_{j}}^{*}(B_{\rr}).
\end{flalign}
To recover \eqref{vej}, we pick any $w\in \mathcal{K}_{\psi,\tilde{v}_{j}}^{*}(B_{\rr})$ and notice that, for $\sigma \in (0,1)$, the function $w_{j,\sigma}:=v_{j}+\sigma(w-v_{j})$ belongs to $\mathcal{K}_{\psi,\tilde{v}_{j}}^{*}(B_{\rr})$, thus it is an admissible competitor in problem \eqref{opj}. By the minimality of $v_{j}$ we have
\begin{flalign}\label{5}
0\le &\sigma^{-1}\int_{B_{\rr}}\left[F_{j}(x,Dw_{j,\sigma})-F_{j}(x,Dv_{j})\right] \ \dx\nonumber \\
=&\int_{B_{\rr}}\left(\int_{0}^{1}\partial_{z} F_{j}(x,Dv_{j}+\lambda\sigma (Dw-Dv_{j})) \ \d \lambda\right)\cdot (Dw-Dv_{j}) \ \dx.
\end{flalign}
Now we can use \eqref{4} to legalize an application of the dominated convergence theorem and send $\sigma \to 0$ in \eqref{5}, the outcome being precisely \eqref{vej}. At this point we fix parameters $0<\frac{\rr}{2}\le t<s\le \rr\le 1$, take a cut-off function $\eta\in C^{1}_{c}(B_{\rr})$ with the following specifics:
\begin{flalign}\label{eta}
\chi_{B_{t}}\le \eta\le \chi_{B_{(s+t)/2}}\quad \mbox{and}\quad \snr{D\eta}\le \frac{4}{s-t}
\end{flalign}
and a vector $h\in \mathbb{R}^{n}$ with $\snr{h}<\frac{1}{1000}\min\left\{s-t,\dist(\partial B_{\rr},\partial \Omega)\right\}$. We look at the map $w_{j}:=v_{j}+\frac{1}{2}\tau_{-h}(\eta^{2}\tau_{h}(v_{j}-\psi))$. By construction, $w_{j}\in W^{1,q}(\Omega)$, condition $\eqref{eta}_{1}$ guarantees that $\left.w_{j}\right|_{\partial B_{\rr}}=\left.\tilde{v}_{j}\right|_{\partial B_{\rr}}$ and
\begin{flalign*}
w_{j}(x)=&v_{j}(x)+\frac{1}{2}\left\{\eta^{2}(x)\left[(v_{j}(x+h)-\psi(x+h))-(v_{j}(x)-\psi(x))\right]\right\}\nonumber \\
&-\frac{1}{2}\left\{\eta^{2}(x-h)\left[(v_{j}(x)-\psi(x))-(v_{j}(x-h)-\psi(x-h))\right]\right\}\nonumber \\
\ge &\frac{1}{2}(1-\eta^{2}(x))v_{j}(x)+\frac{1}{2}\eta^{2}(x)\psi(x)\nonumber \\
&+\frac{1}{2}(1-\eta^{2}(x-h))v_{j}(x)+\frac{1}{2}\eta^{2}(x-h)\psi(x)\ge \psi(x),
\end{flalign*}
therefore $w_{j}\in \mathcal{K}^{*}_{\tilde{v}_{j},\psi}(B_{\rr})$ is an admissible test function in \eqref{vej}. Using the integration by part rule for finite difference operators we obtain
\begin{flalign*}
0\le& -\int_{B_{\rr}}\tau_{h}(\partial_{z} F_{j}(x,Dv_{j}))\cdot D(\eta^{2}\tau_{h}(v_{j}-\psi)) \ \dx\nonumber \\
=&-\int_{B_{\rr}}\eta^{2}\left[\partial_{z} F_{j}(x+h,Dv_{j}(x+h))-\partial_{z} F_{j}(x+h,Dv_{j}(x))\right]\cdot \tau_{h}(Dv_{j}) \ \dx \nonumber \\
&+\int_{B_{\rr}}\eta^{2}\left[\partial_{z} F_{j}(x+h,Dv_{j}(x+h))-\partial_{z} F_{j}(x+h,Dv_{j}(x))\right]\cdot \tau_{h}(D\psi) \ \dx\\
&-2\int_{B_{\rr}}\eta\left[\partial_{z} F_{j}(x+h,Dv_{j}(x+h))-\partial_{z} F_{j}(x+h,Dv_{j}(x))\right]\cdot (\tau_{h}(v_{j}-\psi)D\eta) \ \dx\nonumber \\
&-\int_{B_{\rr}}\eta^{2}\left[\partial_{z} F_{j}(x+h,Dv_{j}(x))-\partial_{z} F_{j}(x,Dv_{j}(x))\right]\cdot \tau_{h}(Dv_{j}) \ \dx\nonumber \\
&+\int_{B_{\rr}}\eta^{2}\left[\partial_{z} F_{j}(x+h,Dv_{j}(x))-\partial_{z} F_{j}(x,Dv_{j}(x))\right]\cdot \tau_{h}(D\psi) \ \dx\nonumber \\
&-2\int_{B_{\rr}}\eta\left[\partial_{z} F_{j}(x+h,Dv_{j}(x))-\partial_{z} F_{j}(x,Dv_{j}(x))\right]\cdot (\tau_{h}(v_{j}-\psi)D\eta) \ \dx\nonumber \\
=:&\mbox{(I)}+\mbox{(II)}+\mbox{(III)}+\mbox{(IV)}+\mbox{(V)}+\mbox{(VI)}.
\end{flalign*}
From $\eqref{assfj}_{3}$ and Lemma \ref{l1}, we readily have
\begin{flalign*}
\mbox{(I)}\le&-\nu\int_{B_{\rr}}(\mu^{2}+\snr{Dv_{j}(x+h)}^{2}+\snr{Dv_{j}(x)}^{2})^{\frac{p-2}{2}}\snr{\tau_{h}(Dv_{j})}^{2} \ \dx\nonumber \\
\le& -c(\nu,p)\int_{B_{\rr}}\eta^{2}\snr{\tau_{h}(V_{\mu,p}(Dv_{j}))}^{2} \ \dx.
\end{flalign*}
Combining $\eqref{4}$, H\"older and Young inequalities, $\eqref{eta}_{1}$, Lemma \ref{l2} and \eqref{34} we obtain
\begin{flalign*}
\snr{\mbox{(II)}}\le&c\int_{B_{\rr}}\eta^{2}\left(1+\snr{Dv_{j}(x+h)}^{2}+\snr{Dv_{j}(x)}^{2}\right)^{\frac{q-1}{2}}\snr{\tau_{h}(D\psi)} \ \dx\nonumber \\
\le &c\left(\int_{B_{(s+t)/2}}(1+\snr{Dv_{j}}^{2})^{\frac{q}{2}}\right)^{\frac{q-1}{q}}\left(\int_{B_{(s+t)/2}}\snr{\tau_{h}(D\psi)}^{q} \ \dx\right)^{\frac{1}{q}}\nonumber \\
\le& c\snr{h}^{\alpha}\left[1+\int_{B_{s}}(1+\snr{Dv_{j}}^{2})^{\frac{q}{2}} \ \dx\right],
\end{flalign*}
where $c=c(L,q,\nr{\psi}_{W^{1+\alpha,q}(\Omega)})$. By \eqref{4}, H\"older and Young inequalities, \eqref{eta}, Lemmas \ref{l3} and \ref{l2} we get
\begin{flalign*}
\snr{(\mbox{III})}\le &c\int_{B_{\rr}}\eta\left(1+\snr{Dv_{j}(x+h)}^{2}+\snr{Dv_{j}(x)}^{2}\right)^{\frac{q-1}{2}}\snr{\tau_{h}(v_{j}-\tilde{\psi_{j}})}\snr{D\eta} \ \dx\nonumber \\
\le &\frac{c}{(s-t)}\left(\int_{B_{(s+t)/2}}(1+\snr{Dv_{j}}^{2})^{\frac{q}{2}} \ \dx\right)^{\frac{q-1}{q}}\left(\int_{B_{(s+t)/2}}\snr{\tau_{h}(v_{j}-\psi)}^{q} \ \dx\right)^{\frac{1}{q}}\nonumber \\
\le &\frac{c\snr{h}}{(s-t)}\left(\int_{B_{s}}(1+\snr{Dv_{j}}^{2})^{\frac{q}{2}} \ \dx\right)^{\frac{q-1}{q}}\left(\int_{B_{s}}\left[1+\snr{Dv_{j}}^{2}+\snr{D\psi}^{2}\right]^{\frac{q}{2}} \ \dx\right)^{\frac{1}{q}}\nonumber \\
\le &\frac{c\snr{h}}{(s-t)}\left[1+\int_{B_{s}}\left[\snr{Dv_{j}}^{q}+\snr{D\psi}^{q}\right] \ \dx\right],
\end{flalign*}
with $c=c(n,L,q)$. Using $\eqref{assfj}_{5}$, H\"older and Young inequalities, $\eqref{eta}_{1}$ and Lemma \ref{l2} we see that
\begin{flalign*}
\snr{\mbox{(IV)}}+\snr{\mbox{(V)}}\le &c\snr{h}^{\alpha}\int_{B_{\rr}}\eta^{2}(1+\snr{Dv_{j}}^{2})^{\frac{q-1}{2}}(\snr{Dv_{j}(x+h)}+\snr{Dv_{j}(x)}) \ \dx\nonumber \\
&+c\snr{h}^{\alpha}\int_{B_{\rr}}\eta^{2}(1+\snr{Dv_{j}}^{2})^{\frac{q-1}{2}}(\snr{D\psi(x+h)}+\snr{D\psi(x)}) \ \dx\nonumber \\
\le &c\snr{h}^{\alpha}\int_{B_{s}}(1+\snr{Dv_{j}}^{2})^{\frac{q}{2}} \ \dx \nonumber \\
&+c\snr{h}^{\alpha}\left(\int_{B_{t}}(1+\snr{Dv_{j}}^{2})^{\frac{q}{2}} \ \dx\right)^{\frac{q-1}{2}}\left(1+\int_{B_{t}}\snr{D\psi}^{q} \ \dx\right)^{\frac{1}{q}}\nonumber \\
\le &c\snr{h}^{\alpha}\left[1+\int_{B_{t}}\left[\snr{Dv_{j}}^{q}+\snr{D\psi}^{q}\right] \ \dx\right],
\end{flalign*}
for $c=c(n,L,q)$. Finally, exploiting $\eqref{assfj}_{4}$, H\"older and Young inequalities, \eqref{eta}, Lemmas \ref{l2} and \ref{l3} we obtain
\begin{flalign*}
\snr{\mbox{(VI)}}\le &\frac{c\snr{h}^{\alpha}}{(s-t)}\left(\int_{B_{s}}(1+\snr{Dv_{j}}^{2})^{\frac{q}{2}} \ \dx\right)^{\frac{q-1}{q}}\left(\int_{B_{(s+t)/2}}\snr{\tau_{h}(v_{j}-\psi)}^{q} \ \dx\right)^{\frac{1}{q}}\nonumber \\
\le &\frac{c\snr{h}^{1+\alpha}}{(s-t)}\left(\int_{B_{s}}(1+\snr{Dv_{j}}^{2})^{\frac{q}{2}} \ \dx\right)\nonumber \\
&+\frac{c\snr{h}^{1+\alpha}}{(s-t)}\left(\int_{B_{s}}(1+\snr{Dv_{j}}^{2})^{\frac{q}{2}} \ \dx\right)^{\frac{q-1}{q}}\left(\int_{B_{s}}\snr{D\psi}^{q} \ \dx\right)^{\frac{1}{q}}\nonumber\\
\le &\frac{c\snr{h}^{1+\alpha}}{(s-t)}\left[1+\int_{B_{s}}\left[\snr{Dv_{j}}^{q}+\snr{D\psi}^{q}\right] \ \dx\right],
\end{flalign*}
where $c=c(n,L,q)$. Merging the content of all the above displays and recalling $\eqref{eta}_{1}$, we can conclude that
\begin{flalign}\label{7}
\int_{B_{t}}\snr{\tau_{h}(V_{\mu,p}(Dv_{j}))}^{2} \ \dx \le \frac{c\snr{h}^{\alpha}}{(s-t)}\left[1+\int_{B_{s}}\left[\snr{Dv_{j}}^{q}+\snr{D\psi}^{q}\right] \ \dx\right],
\end{flalign}
with $c=c(\texttt{data}_{\texttt{q}})$. Now we can invoke Lemma \ref{l4} to get, with \eqref{7},
\begin{flalign*}
\nr{V_{\mu,p}(Dv_{j})}_{W^{\beta,2}(B_{t})}\le \frac{c}{(s-t)^{\frac{n+1+2\beta}{2}}}\left[1+\left(\int_{B_{s}}\left[\snr{Dv_{j}}^{q}+\snr{D\psi}^{q}\right] \ \dx\right)^{\frac{1}{2}}\right],
\end{flalign*}
with $c=c(\texttt{data}_{\texttt{q}},\beta)$, so, by Lemma \ref{l9} we obtain
\begin{flalign}\label{8}
\nr{V_{\mu,p}(Dv_{j})}_{L^{\frac{2n}{n-2\beta}}(B_{t})}\le \frac{c(\texttt{data}_{\texttt{q}},\beta)}{(s-t)^{\tilde{\theta}}}\left[1+\nr{Dv_{j}}_{L^{q}(B_{s})}^{\frac{q}{2}}+\nr{D\psi}_{L^{q}(B_{s})}^{\frac{q}{2}}\right],
\end{flalign}
for all $\beta\in \left(0,\frac{\alpha}{2}\right)$. In \eqref{8}, $\tilde{\theta}=\tilde{\theta}(n,\alpha,\beta)$. We manipulate \eqref{8} in a more convenient way: 
\begin{flalign}\label{9}
\nr{Dv_{j}}_{L^{\frac{np}{n-2\beta}(B_{s})}}\le \frac{c(\texttt{data}_{\texttt{q}},\beta)}{(s-t)^{\bar{\theta}}}\left[1+\nr{Dv_{j}}_{L^{q}(B_{s})}^{\frac{q}{p}}+\nr{D\psi}_{L^{q}(B_{s})}^{\frac{q}{p}}\right],
\end{flalign}
set $\bar{\theta}:=\frac{2\tilde{\theta}}{p}$. Notice that, by \eqref{pq}, for $\beta \in \left[\frac{\alpha n}{2(n+\alpha)},\frac{\alpha}{2}\right)$, there holds that $q<\frac{np}{n-2\beta}$, thus we can apply the interpolation inequality
\begin{flalign}\label{10}
\nr{Dv_{j}}_{L^{q}(B_{s})}\le \nr{Dv_{j}}_{L^{\frac{np}{n-2\beta}}(B_{s})}^{\kappa}\nr{Dv_{j}}_{L^{p}(B_{s})}^{1-\kappa},
\end{flalign}
where $\kappa\in (0,1)$ is derived via the equation
\begin{flalign*}
\frac{1}{q}=\frac{\kappa(n-2\beta)}{np}+\frac{(1-\kappa)}{p}\Rightarrow \kappa=\frac{(q-p)n}{2\beta q}.
\end{flalign*}
Inserting \eqref{10} in \eqref{9} we get
\begin{flalign*}
\nr{Dv_{j}}_{L^{\frac{np}{n-2\beta}}(B_{s})}\le \frac{c(\texttt{data}_{\texttt{q}},\beta)}{(s-t)^{\bar{\theta}}}\left[1+\nr{Dv_{j}}_{L^{\frac{np}{n-2\beta}}(B_{s})}^{\frac{(q-p)n}{2\beta p}}\nr{Dv_{j}}_{L^{p}(B_{s})}^{\frac{(2\beta q-(q-p)n)}{2\beta p}}+\nr{D\psi}_{L^{q}(B_{s})}^{\frac{q}{p}}\right]
\end{flalign*}
so, for $\beta\in \left(\frac{n(q-p)}{2p},\frac{\alpha}{2}\right)$ and \eqref{pq} we see that $\frac{(q-p)n}{2\beta p}<1$. This allows using Young inequality with conjugate exponents $\frac{2\beta p}{(q-p)n}$ and $\frac{2\beta p}{(2\beta p-(q-p)n)}$ to obtain
\begin{flalign}\label{11}
\nr{Dv_{j}}_{L^{\frac{np}{n-2\beta}}(B_{t})}\le \frac{1}{2}\nr{Dv_{j}}_{L^{\frac{np}{n-2\beta}}(B_{s})}+\frac{c(\texttt{data}_{\texttt{q}},\beta)}{(s-t)^{\theta}}\left[1+\nr{D\psi}_{L^{q}(B_{s})}^{\frac{q}{p}}+\nr{Dv_{j}}_{L^{p}(B_{s})}^{\gamma}\right],
\end{flalign}
where we set $\theta:=\frac{2\bar{\theta}\beta p}{(2\beta p-(q-p)n)}$ and $\gamma:=\frac{(2\beta q-(q-p)n)}{(2\beta p-(q-p)n)}$. Since inequality \eqref{11} holds true for all $\frac{\rr}{2}\le t<s\le \rr$, we can use Lemma \ref{l5} to end up with
\begin{flalign}\label{12}
\nr{Dv_{j}}_{L^{\frac{np}{n-2\beta}}(B_{\rr/2})}\le \frac{c(\texttt{data}_{\texttt{q}},\beta)}{\rr^{\theta}}\left[1+\nr{D\psi}_{L^{q}(B_{\rr})}^{\frac{q}{p}}+\nr{Dv_{j}}_{L^{p}(B_{\rr})}^{\gamma}\right],
\end{flalign}
with $\theta=\theta(n,p,q,\alpha,\beta)$ and $\gamma=\gamma(n,p,q,\alpha,\beta)$ as in \eqref{11}. At this stage, we jump back to problem \eqref{opj} and notice that by $\eqref{35}_{2}$, $\tilde{v}_{j}\in \mathcal{K}_{\psi,\tilde{v}_{j}}^{*}(B_{\rr})$. Thus, using the minimality of $v_{j}$ in class $\mathcal{K}^{*}_{\psi,\tilde{v}_{j}}(B_{\rr})$ we get
\begin{flalign}
\nu\int_{B_{\rr}}\snr{Dv_{j}}^{p} \ \dx \stackrel{\eqref{assfj}_{1}}{\le}& \int_{B_{\rr}}F_{j}(x,Dv_{j}) \ \dx \stackrel{\eqref{35}_{2}}{\le} \int_{B_{\rr}}F_{j}(x,D\tilde{v}_{j}) \ \dx\nonumber \\
\stackrel{\eqref{fj}}{=}&\int_{B_{\rr}}F(x,D\tilde{v}_{j}) \ \dx +\frac{\varepsilon_{j}}{q}\int_{B_{\rr}}(1+\snr{D\tilde{v}_{j}}^{2})^{\frac{q}{2}} \ \dx\nonumber \\
\stackrel{\eqref{eps}}{\le}&\int_{B_{\rr}}F(x,D\tilde{v}_{j}) \ \dx +o(j^{-1})\stackrel{\eqref{35}_{3}}{\le}\int_{B_{\rr}}F(x,Dv) \ \dx +1,\label{13}
\end{flalign}
thus 
\begin{flalign}\label{15}
v_{j}\rightharpoonup_{j\to \infty} \tilde{v} \ \ \mbox{in} \ \ W^{1,p}(B_{\rr})\quad \mbox{and}\quad v_{j}\to_{j\to \infty} \tilde{v} \ \ \mbox{in} \ \ L^{p}(B_{\rr}).
\end{flalign}
Merging \eqref{13}, $\eqref{assfj}_{1}$ and \eqref{12} we get
\begin{flalign}\label{14}
\nr{Dv_{j}}_{L^{\frac{np}{n-2\beta}}(B_{\rr/2})}\le \frac{c(\texttt{data}_{\texttt{q}},\beta)}{\rr^{\theta}}\left[1+\nr{D\psi}_{L^{q}(B_{\rr})}^{\frac{q}{p}}+\left(\int_{B_{\rr}}F(x,Dv) \ \dx\right)^{\frac{\gamma}{p}}\right],
\end{flalign}
thus, by $\eqref{15}_{1}$, \eqref{14} and weak lower semicontinuity, we can conclude that
\begin{flalign*}
\nr{D\tilde{v}}_{L^{\frac{np}{n-2\beta}}(B_{\rr/2})}\le& \liminf_{j\to \infty}\nr{Dv_{j}}_{L^{\frac{np}{n-2\beta}}(B_{\rr/2})}\nonumber \\
\le& \frac{c(\texttt{data}_{\texttt{q}},\beta)}{\rr^{\theta}}\left[1+\nr{D\psi}_{L^{q}(B_{\rr})}^{\frac{q}{p}}+\left(\int_{B_{\rr}}F(x,Dv) \ \dx\right)^{\frac{\gamma}{p}}\right].
\end{flalign*}
At this point we only need to show that $\tilde{v}(x)=v(x)$ for a.e. $x\in B_{\rr}$. To do so, we notice that by $\eqref{15}$, $\eqref{35}_{1,2}$ and the weak continuity of the trace operator, there holds that
\begin{flalign}\label{16}
\tilde{v}(x)\ge \psi(x) \ \ \mbox{for a.e.} \ \ x\in B_{\rr}\quad \mbox{and} \quad \left.\tilde{v}\right|_{\partial B_{\rr}}=\left.v\right|_{\partial B_{\rr}}.
\end{flalign}
Moreover, by $\eqref{35}_{3}$, \eqref{eps}, $\eqref{15}_{1}$, weak lover semicontinuity and the minimality of the $v_{j}$'s we have
\begin{flalign}\label{17}
\int_{B_{\rr}}F(x,D\tilde{v}) \ \dx \le &\liminf_{j\to \infty}\int_{B_{\rr}}F(x,Dv_{j}) \ \dx \le \liminf_{j\to \infty}\int_{B_{\rr}}F_{j}(x,Dv_{j}) \ \dx \nonumber\\
\le &\liminf_{j\to \infty}\left[\int_{B_{\rr}}F(x,D\tilde{v}_{j}) \ \dx+o(j^{-1})\right]=\int_{B_{\rr}}F(x,Dv) \ \dx.
\end{flalign}
Collecting estimates \eqref{14} and \eqref{17} and keeping in mind \eqref{sconv} and \eqref{16} we can conclude that $\tilde{v}=v$ a.e. in $B_{\rr}$ and
\begin{flalign}\label{19}
\nr{Dv}_{L^{\frac{np}{n-2\beta}}(B_{\rr/2})}\le \frac{c(\texttt{data}_{\texttt{q}},\beta)}{\rr^{\theta}}\left[1+\int_{B_{\rr}}\left[F(x,Dv)+\snr{D\psi}^{q}\right] \ \dx\right]^{\tilde{\gamma}},
\end{flalign}
with $\theta=\theta(n,p,q,\alpha,\beta)$ and $\tilde{\gamma}=\tilde{\gamma}(n,p,q,\alpha,\beta)$. Recalling that $\beta\in \left(0,\frac{\alpha}{2}\right)$ is arbitrary, using H\"older inequality in \eqref{19} we obtain \eqref{18}, where $d\in \left[1,\frac{np}{n-\alpha}\right)$ is arbitrary. Finally, a standard covering argument renders that $Dv\in L^{d}_{loc}(\Omega,\mathbb{R}^{n})$ and we are done.
\begin{remark}
\emph{For transforming \eqref{8} into \eqref{9}, we implicitely used that, for any map $w\in W^{1,p}(B_{\rr})$, $\rr\in (0,1]$ such that $V_{\mu,p}(Dv)\in L^{2t}(B_{\rr},\mathbb{R}^{n})$ for some $t>0$ there holds}
\begin{flalign}\label{99}
\int_{B_{\rr}}\snr{Dw}^{pt} \ \dx \le c(n,p,t)\left[1+\int_{B_{\rr}}\snr{V_{\mu,p}(Dw)}^{2t} \ \dx\right].
\end{flalign}
\emph{Inequality \eqref{99} is trivial when $p\ge 2$, while for $1<p<2$ we have}
\begin{flalign*}
\int_{B_{\rr}}&\snr{V_{\mu,p}(Dw)}^{2t} \ \dx =\int_{B_{\rr}}\left[(\mu^{2}+\snr{Dw}^{2})^{\frac{p-2}{2}}\snr{Dw}^{2}\right]^{t} \ \dx\ge 2^{\frac{t(p-2)}{2}} \int_{B_{\rr}\cap \{\snr{Dw}\ge \mu\}}\snr{Dw}^{pt} \ \dx,
\end{flalign*}
\emph{therefore}
\begin{flalign*}
\int_{B_{\rr}}\snr{Dw}^{pt} \ \dx \le& \int_{B_{\rr}\cap\{\snr{Dw}\ge \mu\}}\snr{Dw}^{pt} \ \dx +c(n)\mu^{pt}\nonumber \\
\le &2^{\frac{t(2-p)}{2}}\int_{B_{\rr}}\snr{V_{\mu,p}(Dw)}^{2t} \ \dx +c(n),
\end{flalign*}
\emph{where we also used that $\mu\in [0,1]$.}
\end{remark}

\begin{remark}
\emph{The arbitrariety of $\beta$ allows a corresponding choice of $d\in \left[1,\frac{np}{n-\alpha}\right)$, therefore we will translate any dependency of the constants from $\beta$ into the one from $d$, i.e.: $c(\texttt{data}_{\texttt{q}},\beta)$ becomes $c(\texttt{data}_{\texttt{q}},d)$. This justifies the final dependencies of the constant $c$ appearing in \eqref{18}.}
\end{remark}
\subsection{Fractional differentiability}\label{fra}
Let $v\in \mathcal{K}_{\psi,g}(\Omega)$ be the solution of problem \eqref{op}. Combining assumption \eqref{pq} and the outcome of Theorem \ref{t1}, we see that $q<\frac{np}{n-\alpha}$, so, in particular, $Dv\in L^{q}_{loc}(\Omega,\mathbb{R}^{n})$. This means that we no longer need the approximating problems to study the fractional differentiability of $Dv$. In fact, let $B_{\rr}\Subset \Omega$ be any ball with $\rr\in (0,1]$ and notice that, as in Section \ref{ex}, it follows that $v$ is the solution of
\begin{flalign}\label{opq}
\mathcal{K}^{*}_{\psi,v}(B_{\rr})\ni w\mapsto \min \int_{B_{\rr}}F(x,Dw) \ \dx,
\end{flalign}
where $\mathcal{K}^{*}_{\psi,v}(B_{\rr})$ is defined as in \eqref{kjd}, with $v$ instead of $\tilde{v}_{j}$. As for \eqref{vej}, we see that the variational inequality
\begin{flalign}\label{ve}
\int_{B_{\rr}}F(x,Dv)\cdot (Dw-Dv) \ \dx \ge 0
\end{flalign}
holds for all $w\in \mathcal{K}^{*}_{\psi,v}(B_{\rr})$ and the map $w:=v+\tau_{-h}(\eta^{2}\tau_{h}(v-\psi))$ is an admissible test function. Here, $\eta\in C^{1}_{c}(B_{\rr})$ is such that
\begin{flalign*}
\chi_{B_{\rr/2}}\le \eta\le \chi_{B_{3\rr/4}}\quad \mbox{and}\quad \snr{D\eta}\le \frac{4}{\rr}
\end{flalign*}
and $\snr{h}\le \frac{1}{1000}\min\left\{\frac{\rr}{4},\dist(\partial B_{\rr},\partial \Omega)\right\}$. We can repeat exactly the same procedure outlined in Section \ref{pt1} with $v_{j}$, $\tilde{v}_{j}$ both replaced by $v$, to end up with
\begin{flalign}\label{20}
\nr{V_{\mu,p}(Dv)}_{W^{\beta,2}(B_{\rr/2})}\le \frac{c}{\rr^{\frac{n+1+2\beta}{2}}}\left[1+\left(\int_{B_{\rr}}\left[\snr{Dv}^{q}+\snr{D\psi}^{q}\right] \ \dx\right)^{\frac{1}{2}}\right]
\end{flalign}
for all $\beta\in \left(0,\frac{\alpha}{2}\right)$, with $c=c(\texttt{data}_{\texttt{q}},\beta)$. Via a standard covering argument, we can conclude that $V_{\mu,p}(Dv)\in W^{2,\beta}_{loc}(\Omega,\mathbb{R}^{n})$ for all $\beta \in \left(0,\frac{\alpha}{2}\right)$ and the proof is complete.
\section{Proof of Theorem \ref{t2}}\label{pt2}
The proof of Theorem \ref{t2} requires certain assumptions which are stronger that \eqref{assf0}-\eqref{assf}. Precisely, we need a Carath\'eodery integrand $F\colon \Omega\times \mathbb{R}^{n}\to \mathbb{R}$ satisfying
\begin{flalign}\label{assfh}
\begin{cases}
\ z\mapsto F(\cdot,z)\in C^{1}_{loc}(\mathbb{R}^{n})\cap C^{2}_{loc}(\mathbb{R}^{n}\setminus \{0\})\\
\ x\mapsto \partial_{z}F(x,z)\in W^{1,s}(\Omega,\mathbb{R}^{n}) \ \ \mbox{for all} \ z\in \mathbb{R}^{n}\\
\ \nu (\mu^{2}+\snr{z}^{2})^{\frac{p}{2}}\le F(x,z)\le L\left[(\mu^{2}+\snr{z}^{2})^{\frac{p}{2}}+(\mu^{2}+\snr{z}^{2})^{\frac{q}{2}}\right]\\
\ \nu (\mu^{2}+\snr{z}^{2})^{\frac{p-2}{2}}\snr{\xi}^{2}\le \partial_{z}^{2}F(x,z)\xi\cdot \xi \\
\ \snr{\partial^{2}_{z}F(x,z)}\le L\left[(\mu^{2}+\snr{z}^{2})^{\frac{p-2}{2}}+(\mu^{2}+\snr{z}^{2})^{\frac{q-2}{2}}\right]\\
\ \snr{\partial^{2}_{x,z}F(x,z)}\le Lh(x)\left[(\mu^{2}+\snr{z}^{2})^{\frac{p-1}{2}}+(\mu^{2}+\snr{z}^{2})^{\frac{q-1}{2}}\right],\\
\ x\mapsto F(x,0)\equiv \const
\end{cases}
\end{flalign}
for all $x\in \Omega$ and $z,\xi\in \mathbb{R}^{n}$. In \eqref{assfh}, $0<\nu\le L$ are absolute constants, and 
\begin{flalign}\label{assh}
0\le h(\cdot)\in L^{s}(\Omega) \quad \mbox{with} \ s>n,
\end{flalign}
the exponents $(p,q)$ match condition
\begin{flalign}\label{pqh}
1<p<q<p\left(1+\frac{1}{n}-\frac{1}{s}\right).
\end{flalign}
Concerning the obstacle, we shall assume that 
\begin{flalign}\label{obsh}
\psi \in W^{2,\infty}(\Omega).
\end{flalign}
When $1<p<2\le q$ or $1<p<q<2$, we also ask that
\begin{flalign}\label{obsh1}
\mathcal{H}(D\psi):=\max\left\{(\mu^{2}+\snr{D\psi}^{2})^{\frac{p-2}{2}},(\mu^{2}+\snr{D\psi}^{2})^{\frac{q-2}{2}}\right\}\in L^{s}(\Omega),
\end{flalign}
where $s$ is the same as in \eqref{assh}. Notice that the hypotheses considered in Section \ref{re} are trivially satisfied. Moreover, as before, assumption $\eqref{assfh}_{4}$ implies that
\begin{flalign}\label{98}
z\mapsto F(\cdot,z)\quad \mbox{is strictly convex}.
\end{flalign}
We just spend a few lines commenting on the relation between \eqref{pq} and \eqref{pqh}. First, notice that as in \cite[Remark 1.4]{demi}, we directly see that, whenever $x_{1},x_{2}\in \Omega$ and $z\in \mathbb{R}^{n}$, there holds 
\begin{flalign}\label{37}
\snr{\partial_{z} F(x_{1},z)-\partial_{z} F(x_{2},z)}\le c\nr{h}_{L^s(\Omega)}\left[(\mu^2+|z|^2)^{\frac{p-1}{2}}+(\mu^2+|z|^2)^{\frac{q-1}{2}}\right] |x_{1}-x_{2}|^{1-\frac{n}{s}},
\end{flalign}
for $c=c(n,L,p,q,s)$, which is $\eqref{assf}_{4}$ with $\alpha=1-\frac{n}{s}$ and such value of $\alpha$ turns \eqref{pqh} in \eqref{pq}. Therefore, if assumptions \eqref{assfh}-\eqref{obsh} are satisfied and if the solution $v\in \mathcal{K}_{\psi,g}(\Omega)$ of problem \eqref{op} satisfies \eqref{nolavq} for all open subsets $\tilde{\Omega}\Subset \Omega$, then
\begin{flalign*}
Dv\in L^{q}_{loc}(\Omega,\mathbb{R}^{n})\quad \mbox{and}\quad V_{\mu,p}(Dv)\in W^{2,\beta}_{loc}(\Omega,\mathbb{R}^{n}) \ \ \mbox{for all} \ \beta\in \left(0,\frac{1}{2}\left(1-\frac{n}{s}\right)\right)
\end{flalign*}
by Theorem \ref{t1}. Finally, integrating \eqref{37} and using that, by $\eqref{assfh}_{7}$, $F(x_{1},0)=F(x_{2},0)$, we can conclude with
\begin{flalign}\label{38}
\snr{F(x_{1},z)-F(x_{2},z)}\le c\nr{h}_{L^{s}(\Omega)}\left[(\mu^2+\snr{z}^2)^{\frac{p}{2}}+(\mu^{2}+\snr{z}^2)^{\frac{p}{2}}\right]\snr{x_{1}-x_{2}}^{1-\frac{n}{s}},
\end{flalign}
with $c= c(n,L,p,q,s)$.
\begin{remark}
\emph{Assumption $\eqref{assfh}_{7}$ comes essentially for free. In fact, if $F\colon \Omega\times \mathbb{R}^{n}\to \mathbb{R}$ is any integrand as in \eqref{assfh} with $F(x,0)$ not constant for all $x\in \Omega$, then we can consider the shifted function $\tilde{F}(x,z):=F(x,z)-F(x,0)+2L\mu^{p}$. It is then straightforward to check that $\tilde{F}$ matches \eqref{assfh} (with $4L$ instead of $L$) and, by construction, $\tilde{F}(x,0)$ is constantly equal to $2L\mu^{p}$.}
\end{remark}

\begin{remark}
\emph{Assumption \eqref{obsh1} has a significant role only to treat the degenerate case $\mu=0$ when
\begin{flalign}\label{xxx}
\mbox{either}\quad 1<p<q<2\quad \mbox{or}\quad 1<p<2\le q.
\end{flalign}
If $\mu>0$ and \eqref{xxx} holds, we can neglect it up to accept a dependency from $\mu^{-1}$ of the constants appearing in the forthcoming estimates.}
\end{remark}
\subsection{Approximating problems} As in \cite[Section 4]{demi}, we regularize the integrand in \eqref{assfh} and correct its non-standard growth behavior in the following way.  Let $B_{\rr}\subset B_{r}\Subset \Omega $ be two concentric balls with $0<\rr<r\le 1$. We consider a standard family of symmetric mollifiers $\{\rho_{\delta}\}_{\delta}$ for $\delta >0$ such that $\delta < \min\{\dist(\partial B_r,\partial \Omega),1\}/16$, that is
\begin{flalign}\label{molli2}
\rho \in C^{\infty}_{\rm{c}}(B_1(0))\,,\quad \nr{\rho}_{L^{1}(\mathbb{R}^{n})} =1\,,\quad \rho_{\delta}(x):=  \delta^{-n}\rho\left(x/\delta\right)\,,\quad B_{3/4}\subset \textnormal{supp}(\rho).
\end{flalign}
We then define
\eqn{molli1}
$$
F_{\delta}(x,z):=(F*\rho_{\delta})(x,z)=\mint_{B_{1}}\mint_{B_{1}}F(x+\delta \tilde y,z+\delta y) \rho(\tilde{y})  \rho(y)\d\tilde{y} \dy 
\;,
$$
for all $(x,z)\in \bar{B}_{r}\times \mathbb{R}^{n}$. By the very definition in \eqref{molli1} and \eqref{38}, we have
\begin{flalign}\label{39}
F_{\delta}(x,z) \to F(x,z) \quad \mbox{uniformly on compact subsets of $\bar{B}_{r}\times \mathbb{R}^{n}$ as $\delta \to 0$}.
\end{flalign}
We further define 
\begin{flalign}\label{40}
h_{\delta}(x):=(h*\phi_{\delta})(x)=\mint_{B_{1}} h(x+\delta \tilde y)\rho(\tilde{y}) \d\tilde{y}, \quad \mu_{\delta}:=\mu+\delta, \quad H_{\delta}(z):=\mu^{2}_{\delta}+\snr{z}^{2},
\end{flalign}
for $x \in \bar{B}_{r}$ and $z\in \mathbb{R}^{n}$. 
Next, we use that $v$ satisfies \eqref{nolavq} which, by the results in Section \ref{re}, renders a sequence $\{\tilde{v}_{j}\}_{j\in \N} \subset W^{1,p}_{loc}(\Omega)\cap W^{1,q}(B_{r})$ such that 
\begin{flalign}\label{44}
\tilde{v}_{j}\rightharpoonup_{j\to \infty} v \ \mbox{in} \ W^{1,p}_{loc}(\Omega), \quad \tilde{v}_{j}\ge \psi \ \ \mbox{a.e. in}  \ B_{r},\quad  \int_{B_{r}}F(x,D\tilde{v}_{j}) \ \dx\to_{j\to \infty}\int_{B_{r}}F(x,Dv) \ \dx.
\end{flalign}
For simplicity, define
\begin{flalign*}
&\mathcal{H}_{\delta}(D\psi):=\max\left\{H_{\delta}(D\psi)^{\frac{p-2}{2}},H_{\delta}(D\psi)^{\frac{q-2}{2}}\right\}.
\end{flalign*}
Recalling also \eqref{obsh1}, we trivially observe that
\begin{flalign}\label{obsh2}
\nr{\mathcal{H}_{\delta}(D\psi)}_{L^{s}(\Omega)}\le c(n,s)\left[1+\nr{\mathcal{H}(D\psi)}_{L^{s}(\Omega)}\right].
\end{flalign}
We then set, for $(x, z) \in \bar{B}_{r}\times \mathbb{R}^{n}$,
\begin{flalign*}
F_{j,\delta}(x,z):=F_{\delta}(x,z)+\frac{\varepsilon_{j}}{q}(\mu_{\delta}^2+\snr{z}^{2})^{\frac{q}{2}} \quad \mbox{and} \quad \mathcal{F}_{j,\delta}(w, B_{r}):=\int_{B_{r}}F_{j,\delta}(x,Dw)  \dx\;,
\end{flalign*}
with
\begin{flalign}\label{43}
\varepsilon_{j}:=\left(1+j+\nr{D\tilde{v}_{j}}^{q}_{L^{q}(B_{r})}\right)^{-1} \Rightarrow 
\frac{\varepsilon_{j}}{q}\int_{B_{r}}(\mu^{2}+\snr{D\tilde{v}_{j}}^{2})^{\frac{q}{2}}  \dx\to_{j\to \infty}0\;.
\end{flalign}
Finally, we define $m:=\frac{s}{s-2}$. From \eqref{assfh}, \eqref{molli2}, \eqref{molli1} and some convolution arguments, see \cite[Section 4]{demi}, we see that the integrand $F_{j,\delta}$ satisfies
\begin{flalign}\label{assfjd}
\begin{cases}
\ c^{-1}H_{\delta}(z)^{\frac{p}{2}}+ \frac{\varepsilon_{j}}{q}H_{\delta}(z)^{\frac{q}{2}}\le F_{j,\delta}(x,z) \leq  c\left[H_{\delta}(z)^{\frac{p}{2}}+ H_{\delta}(z)^{\frac{q}{2}}\right]\\ 
\ \left[c^{-1}[H_{\delta}(z)]^{\frac{p-2}{2}}+\frac{\eps_j}{c}[H_{\delta}(z)]^{\frac{q-2}{2}}\right]\snr{\xi}^{2}\le \partial_{z}^{2}F_{j,\delta}(x,z)\, \xi\cdot\xi\\ 
\ \snr{\partial^{2}_{z}F_{j,\delta}(x,z)}\le c\left[ H_{\delta}(z)^{\frac{p-2}{2}}+   H_{\delta}(z)^{\frac{q-2}{2}}\right]\\ 
\ \snr{\partial_{x,z}^{2} F_{j,\delta}(x,z)}\le c h_{\delta}(x)\left[H_{\delta}(z)^{\frac{p-1}{2}}+ H_{\delta}(z)^{\frac{q-1}{2}} \right] \\ 
\ \snr{\partial_{x,z}^{2}F_{j,\delta}(x,z)}\le c\nr{h_{\delta}}_{L^{\infty}(B_{r})}\left[H_{\delta}(z)^{\frac{p-1}{2}}+H_{\delta}(z)^{\frac{q-1}{2}}\right]\\
\ \nr{h_{\delta}}_{L^{s}(B_{r})}\le \nr{h}_{L^{s}(\Omega)},
\end{cases}
\end{flalign}
for all $x\in B_{r}$ and $z,\xi\in \mathbb{R}^{n}$ with $c=c(n,\nu,L,p,q)$. We stress that $\eqref{assfjd}_{3}$ in particular implies strict convexity and the monotonicity inequality
\begin{flalign}\label{53}
(\partial_{z}F_{j,\delta}(x,z_{1})&-\partial_{z}F_{j,\delta}(x,z_{2}))\cdot (z_{1}-z_{2})\nonumber \\
\ge& c(n,\nu,L,p,q)\left[\snr{V_{\mu_{\delta},p}(z_{1})-V_{\mu_{\delta},p}(z_{2})}^{2}+\varepsilon_{j}\snr{V_{\mu_{\delta},q}(z_{1})-V_{\mu_{\delta},q}(z_{2})}^{2}\right]\ge 0,
\end{flalign}
see also Lemma \ref{l1}. Let us consider the obstacle problem
\begin{flalign}\label{opjd}
\mathcal{K}_{\psi,\tilde{v}_{j}}^{*}(B_{r})\ni w\mapsto \min \mathcal{F}_{j,\delta}(w,B_{r}),
\end{flalign}
where $\mathcal{K}_{\psi,\tilde{v}_{j}}^{*}(B_{r})$ is the same as in \eqref{kjd}. By direct methods (cf. Section \ref{ex}) we know that there exists a unique solution $v_{j,\delta}\in \mathcal{K}^{*}_{\tilde{v}_{j},\psi}(B_{r})$ of problem \eqref{opjd}, satisfying the variational inequality
\begin{flalign}\label{46}
\int_{B_{r}}\partial_{z} F_{j,\delta}(x,Dv_{j,\delta})\cdot (Dw-Dv_{j,\delta}) \ \dx \ge 0 \quad \mbox{for all} \ w\in \mathcal{K}_{\tilde{v}_{j},\psi}^{*}(B_{r}).
\end{flalign}
Moreover, recalling the discussion in Section \ref{ma}, $v_{j,\delta}\in \mathcal{K}_{\psi,\tilde{v}_{j}}^{*}(B_{r})$ is a local minimizer of the variational integral $\mathcal{F}_{j,\delta}$ with obstacle constraint, thus assumptions \eqref{assfjd}, \eqref{obsh}, \eqref{44} together with Proposition \ref{regst} assure that
\begin{flalign}\label{48}
v_{j,\delta}\in W^{1,\infty}_{loc}(B_{r})\cap W^{2,2}_{loc}(B_{r})\quad \mbox{and}\quad \partial_{z}F_{j,\delta}(\cdot,Dv_{j,\delta})\in W^{1,2}_{loc}(B_{r},\mathbb{R}^{n}).
\end{flalign}
\subsection{Linearization}
We aim to recover an integral identity from the variational inequality \eqref{46}. To do so, we follow the arguments in \cite{fuli,fumi} and pick a cut-off function $\eta\in C^{1}_{c}(B_{r})$ so that $0\le \eta(x)\le 1$ for all $x\in B_{r}$ and, for $\varsigma\in (0,\infty)$, we take a function $h_{\varsigma}\in C^{1}(\mathbb{R})$ satisfying 
\begin{flalign}\label{54}
\begin{cases}
\ 0\le h_{\varsigma}(t)\le 1,\\
\ h'_{\varsigma}(t)\le 0 \ \ \mbox{for all} \ t\in \mathbb{R} \\
\ h_{\varsigma}(t)=1 \ \ \mbox{if} \ t\in (-\infty,\varsigma)\\ 
\ h_{\varsigma}(t)=0 \ \ \mbox{if} \ t\in (2\varsigma,+\infty).
\end{cases}
\end{flalign}
The map $w_{j,\delta}^{\varsigma}:=v_{j,\delta}+\eta h_{\varsigma}(v_{j,\delta}-\psi)$ clearly belongs to $\mathcal{K}_{\tilde{v}_{j},\psi}^{*}(B_{r})$, thus it is an admissible test in \eqref{46}. We then get
\begin{flalign*}
\int_{B_{r}}\partial_{z}F_{j,\delta}(x,Dv_{j,\delta})\cdot D(\eta h_{\varsigma}(v_{j,\delta}-\psi)) \ \dx \ge 0,
\end{flalign*}
so, by Riesz representation theorem there exists a non-negative Radon measure $\lambda_{j,\delta}$ such that
\begin{flalign}\label{50}
\int_{B_{r}}\partial_{z}F_{j,\delta}(x,Dv_{j,\delta})\cdot D(\eta h_{\varsigma}(v_{j,\delta}-\psi)) \ \dx =\int_{B_{r}}\eta \d \lambda_{j,\delta}.
\end{flalign}
Notice that, as shown in \cite[Section 3]{fuli}, $\lambda_{j,\delta}$ does not depend on $\varsigma$. Let us find a suitable representative for the measure $\lambda_{j,\delta}$. From \eqref{50}, \eqref{53} and $\eqref{54}_{2}$ we estimate
\begin{flalign*}
\int_{B_{r}}&\partial_{z}F_{j,\delta}(x,Dv_{j,\delta})\cdot D(\eta(h_{\varsigma}(v_{j,\delta}-\psi))) \ \dx =\int_{B_{r}}h_{\varsigma}(v_{j,\delta}-\psi)\partial_{z}F_{j,\delta}(x,Dv_{j,\delta})\cdot D\eta \ \dx \nonumber \\
&+\int_{B_{r}}\eta h'_{\varsigma}(v_{j,\delta}-\psi)\partial_{z}F_{j,\delta}(x,Dv_{j,\delta})\cdot (Dv_{j,\delta}-D\psi) \ \dx \nonumber \\
\le &\int_{B_{r}}h_{\varsigma}(v_{j,\delta}-\psi)\partial_{z}F_{j,\delta}(x,Dv_{j,\delta})\cdot D\eta \ \dx\nonumber \\
&+\int_{B_{r}}\eta h'_{\varsigma}(v_{j,\delta}-\psi)\partial_{z}F_{j,\delta}(x,D\psi)\cdot (Dv_{j,\delta}-D\psi) \ \dx\nonumber \\
=&\int_{B_{r}}h_{\varsigma}(v_{j,\delta}-\psi)\left[\partial_{z}F_{j,\delta}(x,Dv_{j,\delta})-\partial_{z}F_{j,\delta}(x,D\psi)\right]\cdot D\eta \ \dx\nonumber \\
&+\int_{B_{r}}\partial_{z}F_{j,\delta}(x,D\psi) \cdot D(\eta h_{\varsigma}(v_{j,\delta}-\psi)) \ \dx =:\mbox{(I)}+\mbox{(II)}.
\end{flalign*}
Set $\mathcal{S}(B_{r}):=\left\{x\in B_{r}\colon v_{j,\delta}(x)=\psi(x)\right\}$. Using the position in \eqref{54}, we get that
\begin{flalign}\label{55}
\mbox{(I)}\to_{\varsigma\to 0}\int_{\mathcal{S}(B_{r})}\left[\partial_{z}F_{j,\delta}(x,Dv_{j,\delta})-\partial_{z}F_{j,\delta}(x,D\psi)\right]\cdot D\eta \ \dx=0,
\end{flalign}
since $Dv_{j,\delta}=D\psi$ on $\mathcal{S}(B_{r})$. Concerning term (II), by \eqref{obsh} and \eqref{48} we can integrate by parts, thus getting
\begin{flalign}\label{56}
\mbox{(II)}=&-\int_{B_{r}}\diver(\partial_{z}F_{j,\delta}(x,D\psi))(\eta h_{\varsigma}(v_{j,\delta}-\psi)) \ \dx\nonumber \\
&\to_{\varsigma\to 0}-\int_{\mathcal{S}(B_{r})}\diver(\partial_{z}F_{j,\delta}(x,D\psi)) \eta \ \dx.
\end{flalign}
Merging \eqref{50}, \eqref{55} and \eqref{56} we obtain
\begin{flalign*}
\int_{B_{r}}\eta \d \lambda_{j,\delta}\le \int_{B_{r}}\chi_{\mathcal{S}(B_{r})}(-\diver(\partial_{z}F_{j,\delta}(x,D\psi))\eta) \ \dx,
\end{flalign*}
for all $\eta\in C^{1}_{c}(B_{r})$ such that $0\le \eta\le 1$. This implies that  
\begin{flalign*}
\chi_{\mathcal{S}(B_{r})}(-\diver(\partial_{z}F_{j,\delta}(x,D\psi)))\ge 0\quad \mbox{a.e. in} \ B_{r} 
\end{flalign*}
and that there exists a density function $\theta_{j,\delta}\colon B_{r}\to [0,1]$ such that
\begin{flalign}\label{70}
\d\lambda_{j,\delta}=\theta_{j,\delta}(x)\chi_{\mathcal{S}(B_{r})}(-\diver(\partial_{z}F_{j,\delta}(x,D\psi)))\dx.
\end{flalign}
Set $f_{j,\delta}(x):=\theta_{j,\delta}(x)\chi_{\mathcal{S}(B_{r})}(-\diver(\partial_{z}F_{j,\delta}(x,D\psi)))$. Notice that by \eqref{obsh}, $f_{j,\delta}$ exists almost everywhere in $B_{r}$, thus we can compute
\begin{flalign*}
\diver(\partial_{z}F_{j,\delta}(x,D\psi))=\sum_{i=1}^{n}\left[\partial^{2}_{x_{i},z_{i}}F_{j,\delta}(x,D\psi)+\sum_{l=1}^{n}\partial_{z_{l},z_{i}}^{2}F(x,D\psi)D^{2}_{x_{l},x_{i}}\psi\right], 
\end{flalign*}
so by $\eqref{assfjd}_{3,4}$ there holds that
\begin{flalign}\label{60}
\snr{f_{j,\delta}(x)}\le&\snr{\diver(\partial_{z}F_{j,\delta}(x,D\psi)))}\le c\nr{h_{\delta}}_{L^{\infty}(B_{r})}\left[H_{\delta}(D\psi)^{\frac{p-1}{2}}+H_{\delta}(D\psi)^{\frac{q-1}{2}}\right]\nonumber \\
&+c\left[H_{\delta}(D\psi)^{\frac{p-2}{2}}+H_{\delta}(D\psi)^{\frac{q-2}{2}}\right]\nonumber \\
\le& c\nr{h_{\delta}}_{L^{\infty}(B_{r})}\left[H_{\delta}(\nr{D\psi}_{L^{\infty}(B_{r})})^{\frac{p-1}{2}}+H_{\delta}(\nr{D\psi}_{L^{\infty}(B_{r})})^{\frac{q-1}{2}}\right]\nonumber \\
&+c\left[H_{\delta}(\nr{D\psi}_{L^{\infty}(B_{r})})^{\frac{p-2}{2}}+H_{\delta}(\nr{D\psi}_{L^{\infty}(B_{r})})^{\frac{q-2}{2}}+\mu_{\delta}^{p-2}+\mu_{\delta}^{q-2}\right],
\end{flalign}
where $c=c(n,\nu,L,p,q)$, while, by $\eqref{assfjd}_{3,5}$, we have
\begin{flalign}\label{62}
\snr{f_{j,\delta}(x)}\le& ch_{\delta}(x)\left[H_{\delta}(\nr{D\psi}_{L^{\infty}(B_{r})})^{\frac{p-1}{2}}+H_{\delta}(\nr{D\psi}_{L^{\infty}(B_{r})})^{\frac{q-1}{2}}\right]\nonumber \\
&+c\nr{D^{2}\psi}_{L^{\infty}(B_{r})}\max\left\{H_{\delta}(D\psi)^{\frac{p-2}{2}},H_{\delta}(D\psi)^{\frac{q-2}{2}}\right\},
\end{flalign}
with $c=c(n,\nu,L,p,q)$. Estimate \eqref{60} implies that 
\begin{flalign}\label{61}
f_{j,\delta}\in L^{\infty}(B_{r}),
\end{flalign}
and, by \eqref{62}, \eqref{assh}, \eqref{obsh1}, $\eqref{assfjd}_{6}$, and \eqref{obsh2} we see that
\begin{flalign}\label{64}
\nr{f_{j,\delta}}_{L^{s}(B_{r})}\le c\left[1+\nr{h}_{L^{s}(\Omega)}+\nr{\mathcal{H}(D\psi)}_{L^{s}(\Omega)}\right]=:c(n,\nu,L,p,q,\nr{\psi}_{W^{2,\infty}(\Omega)})\mathcal{A}_{h,\psi}.
\end{flalign}
This means that the $f_{j,\delta}$'s have uniformly bounded $L^{s}$-norm. Once identified $\lambda_{j,\delta}$ we turn back to \eqref{50}, which, as in \cite{fu}, implies that
\begin{flalign}\label{57}
\int_{B_{r}}\partial_{z}F_{j,\delta}(x,Dv_{j,\delta})\cdot D\eta \ \dx=\int_{B_{r}}\eta \ \d\lambda_{j,\delta}
\end{flalign}
for all $\eta\in C^{1}_{c}(B_{r})$ such that $0\le \eta\le 1$ on $B_{r}$. Now \eqref{57}, \eqref{70}, \eqref{61} and standard density arguments lead to
\begin{flalign}\label{58}
\int_{B_{r}}\partial_{z}F_{j,\delta}(x,Dv_{j,\delta})\cdot Dw \ \dx=\int_{B_{r}}f_{j,\delta}w\d x\quad \mbox{for all} \ w\in W^{1,q}_{0}(B_{r}).
\end{flalign}

\subsection{Caccioppoli inequality}
By virtue of \eqref{48}, we can differentiate equation \eqref{58} and sum over $i\in \{1,\cdots, n\}$ to obtain
\begin{flalign}\label{71}
\sum_{i=1}^{n}\int_{B_{r}}\left[\partial^{2}_{z}F_{j,\delta}(x,Dv_{j,\delta})D(D_{i}v_{j,\delta})+\partial^{2}_{x_{i},z}F_{j,\delta}(x,Dv_{j,\delta})\right]\cdot Dw \ \dx =-\sum_{i=1}^{n}\int_{B_{r}}f_{j,\delta}D_{i}w \ \dx,
\end{flalign}
which holds for all $w\in W^{1,2}(B_{r})$ with $\supp(w)\Subset B_{r}$. We let $\eta \in C^{1}_{c}(B_{r})$ be any non-negative map, $\gamma\ge 0$ a fixed number and set $w^{i}_{j,\delta}:=\eta^{2}H_{\delta}(Dv_{j,\delta})^{\gamma}D_{i}v_{j,\delta}$. A straightforward computation shows that
\begin{flalign*}
Dw^{i}_{j,\delta}=&\gamma\eta^{2}H_{\delta}(Dv_{j,\delta})^{\gamma-1}D_{i}v_{j,\delta}D(H_{\delta}(Dv_{j,\delta}))+\eta^{2}H_{\delta}(Dv_{j,\delta})^{\gamma}D(D_{i}v_{j,\delta})\nonumber \\
&+2H_{\delta}(Dv_{j,\delta})^{\gamma}D_{i}v_{j,\delta}D\eta,
\end{flalign*}
so, again by \eqref{48}, $w^{i}_{j,\delta}$ is admissible in \eqref{71}. We can rewrite \eqref{71} as
\begin{flalign}\label{72}
0=\mbox{(I)}_{z}+\mbox{(II)}_{z}+\mbox{(III)}_{z}+\mbox{(I)}_{x}+\mbox{(II)}_{x}+\mbox{(III)}_{x}+\mbox{(I)}_{\psi}+\mbox{(II)}_{\psi}+\mbox{(III)}_{\psi}
\end{flalign}
where the terms indexed with $x$ (resp. $\psi$) denote the ones stemming from those in \eqref{71} containing $\partial_{x,z}^{2}F_{j,\delta}$ (resp. $f_{j,\delta}$). Since
\begin{flalign*}
D(H_{\delta}(Dv_{j,\delta}))=2\sum_{l=1}^{n} D(D_{l}v_{j,\delta})D_{l}v_{j,\delta},
\end{flalign*}
with $\eqref{assfjd}_{2}$ we estimate
\begin{flalign*}
\mbox{(I)}_{z}&+\mbox{(II)}_{z}=\gamma\sum_{i=1}^{n}\int_{B_{r}}\eta^{2}H_{\delta}(Dv_{j,\delta})^{\gamma-1}\partial^{2}_{z}F_{j,\delta}(x,Dv_{j,\delta})D(D_{i}v_{j,\delta})\cdot D_{i}v_{j,\delta}D(H_{\delta}(Dv_{j,\delta})) \ \dx\nonumber \\
&+\int_{B_{r}}\eta^{2}H_{\delta}(Dv_{j,\delta})^{\gamma}\partial^{2}_{z}F_{j,\delta}(x,Dv_{j,\delta})\left(\sum_{i=1}^{n}D(D_{i}v_{j,\delta})\cdot D(D_{i}v_{j,\delta})\right) \ \dx\nonumber \\
\ge & \frac{\gamma}{c}\int_{B_{r}}\eta^{2}H_{\delta}(Dv_{j,\delta})^{\frac{p-4}{2}+\gamma}\snr{D(H_{\delta}(Dv_{j,\delta}))}^{2} \ \dx+\frac{1}{c}\int_{B_{r}}\eta^{2}H_{\delta}(Dv_{j,\delta})^{\frac{p-2}{2}+\gamma}\snr{D^{2}v_{j,\delta}}^{2} \ \dx.
\end{flalign*}
From $\eqref{assfjd}_{3}$, H\"older and Young inequalities we have
\begin{flalign*}
\snr{\mbox{(III)}_{z}}=&2\left | \ \sum_{i=1}^{n}\int_{B_{r}}\eta H_{\delta}(Dv_{j,\delta})^{\gamma}\partial^{2}_{z}F_{j,\delta}(x,Dv_{j,\delta})D(D_{i}v_{j,\delta})\cdot D_{i}v_{j,\delta} D\eta\ \dx\ \right|\nonumber \\
\le &c\int_{B_{r}}\eta H_{\delta}(Dv_{j,\delta})^{\gamma}\left[H_{\delta}(Dv_{j,\delta})^{\frac{p-2}{2}}+H_{\delta}(Dv_{j,\delta})^{\frac{q-2}{2}}\right]\snr{D^{2}v_{j,\delta}}\snr{Dv_{j,\delta}}\snr{D\eta} \ \dx\nonumber \\
\le&\frac{\sigma}{c}\int_{B_{r}}\eta^{2}H_{\delta}(Dv_{j,\delta})^{\frac{p-2}{2}+\gamma}\snr{D^{2}v_{j,\delta}}^{2} \ \dx\nonumber \\
&+\frac{c}{\sigma}\int_{B_{r}}\snr{D\eta}^{2}\left[H_{\delta}(Dv_{j,\delta})^{\frac{p}{2}+\gamma}+H_{\delta}(Dv_{j,\delta})^{q-\frac{p}{2}+\gamma}\right] \ \dx\nonumber \\
\le &\frac{\sigma}{c}\int_{B_{r}}\eta^{2}H_{\delta}(Dv_{j,\delta})^{\frac{p-2}{2}+\gamma}\snr{D^{2}v_{j,\delta}}^{2} \ \dx\nonumber \\
&+\frac{c}{\sigma}\left(\int_{B_{r}}\snr{D\eta}^{2m}\left[1+H_{\delta}(Dv_{j,\delta})^{m\left(q-\frac{p}{2}+\gamma\right)}\right] \ \dx\right)^{\frac{1}{m}},
\end{flalign*}
with $c=c(n,\nu,L,p,q,s)$. By $\eqref{assfjd}_{4,6}$, H\"older and Young inequalities we see that
\begin{flalign*}
\snr{(\mbox{I})_{x}}=&\gamma\left | \ \sum_{i=1}^{n}\int_{B_{r}}\eta^{2}H_{\delta}(Dv_{j,\delta})^{\gamma-1}\partial_{x_{i},z}^{2}F_{j,\delta}(x,Dv_{j,\delta})D_{i}v_{j,\delta}D(H_{\delta}(Dv_{j,\delta})) \ \dx \ \right |\nonumber \\
\le&c\gamma \int_{B_{r}}h_{\delta}(x)\left[H_{\delta}(Dv_{j,\delta})^{\frac{p-2}{2}+\gamma}+H_{\delta}(Dv_{j,\delta})^{\frac{q-2}{2}+\gamma}\right]\snr{D(H_{\delta}(Dv_{j,\delta}))} \ \dx\nonumber \\
\le &\sigma\frac{\gamma}{c}\int_{B_{r}}\eta^{2}H_{\delta}(Dv_{j,\delta})^{\frac{p-4}{2}+\gamma}\snr{D(H_{\delta}(Dv_{j,\delta}))}^{2} \ \dx\nonumber \\
&+c\frac{\gamma}{\sigma}\int_{B_{r}}\eta^{2}h_{\delta}(x)^{2}\left[H_{\delta}(Dv_{j,\delta})^{\frac{p}{2}+\gamma}+H_{\delta}(Dv_{j,\delta})^{q-\frac{p}{2}+\gamma}\right] \ \dx \nonumber \\
\le &\sigma\frac{\gamma}{c}\int_{B_{r}}\eta^{2}H_{\delta}(Dv_{j,\delta})^{\frac{p-4}{2}+\gamma}\snr{D(H_{\delta}(Dv_{j,\delta}))}^{2} \ \dx\nonumber \\
&+c\frac{\gamma}{\sigma}\nr{h}_{L^{s}(\Omega)}^{2}\left(\int_{B_{r}}\eta^{2m}\left[H_{\delta}(Dv_{j,\delta})^{m\left(\frac{p}{2}+\gamma\right)}+H_{\delta}(Dv_{j,\delta})^{m\left(q-\frac{p}{2}+\gamma\right)}\right] \ \dx\right)^{\frac{1}{m}},
\end{flalign*}
where $c=c(n,\nu,L,p,q,s)$. In an analogous fashion we also bound
\begin{flalign*}
\snr{\mbox{(II)}_{x}}=&\left| \ \sum_{i=1}^{n}\int_{B_{r}}\eta^{2}H_{\delta}(Dv_{j,\delta})^{\gamma}\partial_{x_{i},z}F_{j,\delta}(x,Dv_{j,\delta})\cdot D(D_{i}v_{j,\delta}) \ \dx\ \right |\nonumber \\
\le &c\int_{B_{r}}\eta^{2}h_{\delta}(x)\left[H_{\delta}(Dv_{j,\delta})^{\frac{p-1}{2}+\gamma}+H_{\delta}(Dv_{j,\delta})^{\frac{q-1}{2}+\gamma}\right]\snr{D^{2}v_{j,\delta}} \ \dx\nonumber \\
\le&\frac{\sigma}{c}\int_{B_{r}}\eta^{2}H_{\delta}(Dv_{j,\delta})^{\frac{p-2}{2}+\gamma}\snr{D^{2}v_{j,\delta}}^{2} \ \dx\nonumber \\
&+\frac{c}{\sigma}\int_{B_{r}}\eta^{2}h_{\delta}(x)^{2}\left[H_{\delta}(Dv_{j,\delta})^{\frac{p}{2}+\gamma}+H_{\delta}(Dv_{j,\delta})^{q-\frac{p}{2}+\gamma}\right] \ \dx\nonumber \\
\le&\frac{\sigma}{c}\int_{B_{r}}\eta^{2}H_{\delta}(Dv_{j,\delta})^{\frac{p-2}{2}+\gamma}\snr{D^{2}v_{j,\delta}}^{2} \ \dx\nonumber \\
&+\frac{c}{\sigma}\nr{h}_{L^{s}(\Omega)}\left(\int_{B_{r}}\eta^{2m}\left[H_{\delta}(Dv_{j,\delta})^{m\left(\frac{p}{2}+\gamma\right)}+H_{\delta}(Dv_{j,\delta})^{m\left(q-\frac{p}{2}+\gamma\right)}\right] \ \dx\right)^{\frac{1}{m}}
\end{flalign*}
and
\begin{flalign*}
\snr{\mbox{(III)}_{x}}=&2\left| \ \sum_{i=1}^{n} \int_{B_{r}}\eta H_{\delta}(Dv_{j,\delta})^{\gamma}\partial_{x_{i},z}F_{j,\delta}(x,Dv_{j,\delta})\cdot D_{i}v_{j,\delta}D\eta \ \dx \ \right |\nonumber \\
\le &c\int_{B_{r}}\eta h_{\delta}(x)\left[H_{\delta}(Dv_{j,\delta})^{\frac{p}{2}+\gamma}+H_{\delta}(Dv_{j,\delta})^{\frac{q}{2}+\gamma}\right]\snr{D\eta} \ \dx \nonumber \\
\le &c\int_{B_{r}}\snr{D\eta}^{2}\left[H_{\delta}(Dv_{j,\delta})^{\frac{p}{2}+\gamma}+H_{\delta}(Dv_{j,\delta})^{\frac{q}{2}+\gamma}\right] \ \dx\nonumber \\
&+c\nr{h}_{L^{s}(\Omega)}^{2}\left(\int_{B_{r}}\eta^{2m}\left[H_{\delta}(Dv_{j,\delta})^{m\left(\frac{p}{2}+\gamma\right)}+H_{\delta}(Dv_{j,\delta})^{m\left(\frac{q}{2}+\gamma\right)}\right] \ \dx\right)^{\frac{1}{m}}\nonumber \\
\le &c\left(1+\nr{h}_{L^{s}(\Omega)}^{2}\right)\left(\int_{B_{r}}(\eta^{2m}+\snr{D\eta}^{2m})\left[1+H_{\delta}(Dv_{j,\delta})^{m\left(q-\frac{p}{2}+\gamma\right)}\right] \ \dx \right)^{\frac{1}{m}}.
\end{flalign*}
In the previous two displays, $c=c(n,\nu,L,p,q,s)$. Finally, by means of \eqref{64}, \eqref{pqh}, H\"older and Young inequalities we control
\begin{flalign*}
\snr{\mbox{(I)}_{\psi}}:=&\gamma\left | \ \sum_{i=1}^{n}\int_{B_{r}}\eta^{2}H_{\delta}(Dv_{j,\delta})^{\gamma-1}f_{j,\delta}D_{i}v_{j,\delta}D(H_{\delta}(Dv_{j,\delta})) \ \dx \ \right |\nonumber \\
\le &\sigma\frac{\gamma}{c}\int_{B_{r}}\eta^{2}H_{\delta}(Dv_{j,\delta})^{\frac{p-4}{2}+\gamma}\snr{D(H_{\delta}(Dv_{j,\delta}))}^{2} \ \dx\nonumber \\
&+c\frac{\gamma}{\sigma}\int_{B_{r}}\eta^{2}f_{j,\delta}^{2}H_{\delta}(Dv_{j,\delta})^{\gamma+1-\frac{p}{2}} \ \dx \nonumber \\
\le &\sigma\frac{\gamma}{c}\int_{B_{r}}\eta^{2}H_{\delta}(Dv_{j,\delta})^{\frac{p-4}{2}+\gamma}\snr{D(H_{\delta}(Dv_{j,\delta}))}^{2} \ \dx\nonumber \\ &+c\frac{\gamma}{\sigma}\nr{f_{j,\delta}}_{L^{s}(\Omega)}^{2}\left(\int_{B_{r}}\eta^{2m}H_{\delta}(Dv_{j,\delta})^{m\left(1-\frac{p}{2}+\gamma\right)} \ \dx\right)^{\frac{1}{m}}\nonumber \\
\le &\sigma\frac{\gamma}{c}\int_{B_{r}}\eta^{2}H_{\delta}(Dv_{j,\delta})^{\frac{p-4}{2}+\gamma}\snr{D(H_{\delta}(Dv_{j,\delta}))}^{2} \ \dx \nonumber \\
&+c\frac{\gamma}{\sigma}\left(\int_{B_{r}}\eta^{2m}\left[1+H_{\delta}(Dv_{j,\delta})^{m\left(q-\frac{p}{2}+\gamma\right)}\right] \ \dx\right)^{\frac{1}{m}}.
\end{flalign*}
Similarly we have
\begin{flalign*}
\snr{\mbox{(II)}_{\psi}}=&\left| \ \sum_{i=1}^{n}\int_{B_{r}}\eta^{2}f_{j,\delta}H_{\delta}(Dv_{j,\delta})^{\gamma}D(D_{i}v_{j,\delta}) \ \dx \ \right|\nonumber \\
\le &\frac{\sigma}{c}\int_{B_{r}}\eta^{2}H_{\delta}(Dv_{j,\delta})^{\frac{p-2}{2}+\gamma}\snr{Dv_{j,\delta}}^{2} \ \dx\nonumber \\
&+\frac{c}{\sigma}\int_{B_{r}}\eta^{2}f_{j,\delta}^{2}H_{\delta}(Dv_{j,\delta})^{1-\frac{p}{2}+\gamma} \ \dx\nonumber \\
\le &\frac{\sigma}{c}\int_{B_{r}}\eta^{2}H_{\delta}(Dv_{j,\delta})^{\frac{p-2}{2}+\gamma}\snr{Dv_{j,\delta}}^{2} \ \dx\nonumber \\
&+c\frac{\gamma}{\sigma}\nr{f_{j,\delta}}_{L^{s}(\Omega)}^{2}\left(\int_{B_{r}}\eta^{2m}H_{\delta}(Dv_{j,\delta})^{m\left(1-\frac{p}{2}+\gamma\right)} \ \dx\right)^{\frac{1}{m}}\nonumber \\
\le &\frac{\sigma}{c}\int_{B_{r}}\eta^{2}H_{\delta}(Dv_{j,\delta})^{\frac{p-2}{2}+\gamma}\snr{Dv_{j,\delta}}^{2} \ \dx\nonumber \\
&+c\frac{\gamma}{\sigma}\left(\int_{B_{r}}\eta^{2m}\left[1+H_{\delta}(Dv_{j,\delta})^{m\left(q-\frac{p}{2}+\gamma\right)}\right] \ \dx\right)^{\frac{1}{m}}
\end{flalign*}
and 
\begin{flalign*}
\snr{\mbox{(III)}_{\psi}}=&2\left| \ \sum_{i=1}^{n}\int_{B_{r}}\eta H_{\delta}(Dv_{j,\delta})^{\gamma}f_{j,\delta}D_{i}v_{j,\delta}D\eta \ \dx \ \right|\nonumber \\
\le &c\int_{B_{r}}\snr{D\eta}^{2}H_{\delta}(Dv_{j,\delta})^{\gamma+\frac{1}{2}} \ \dx +c\nr{f_{j,\delta}}_{L^{s}(B_{r})}^{2}\left(\int_{B_{r}}\eta^{2m}H_{\delta}(Dv_{j,\delta})^{m\left(\gamma+\frac{1}{2}\right)} \ \dx\right)^{\frac{1}{m}}\nonumber \\
\le &c\left(\int_{B_{r}}(\eta^{2m}+\snr{D\eta}^{2m})\left[1+H_{\delta}(Dv_{j,\delta})^{m\left(q-\frac{p}{2}+\gamma\right)}\right] \ \dx\right)^{\frac{1}{m}},
\end{flalign*}
where we also used that $q-\frac{p}{2}>\frac{1}{2}$, being $p>1$. In the above three displays, $c=c(\texttt{data}_{\infty})$. All in all, we got
\begin{flalign}\label{74}
\gamma\int_{B_{r}}&\eta^{2}H_{\delta}(Dv_{j,\delta})^{\frac{p-4}{2}+\gamma}\snr{D(H_{\delta}(Dv_{j,\delta}))}^{2} \ \dx +\int_{B_{r}}\eta^{2}H_{\delta}(Dv_{j,\delta})^{\frac{p-2}{2}+\gamma}\snr{D^{2}v_{j,\delta}}^{2} \ \dx \nonumber \\
\le &c(\texttt{data}_{\infty})(1+\gamma)\left(\int_{B_{r}}(\eta^{2m}+\snr{D\eta}^{2m})\left[1+H_{\delta}(Dv_{j,\delta})^{m\left(q-\frac{p}{2}+\gamma\right)}\right] \ \dx\right)^{\frac{1}{m}}.
\end{flalign}
By \eqref{74}, Sobolev embedding theorem combined with the elementary inequality $(t^{l}+1)\le (t+1)^{l}$ for $t\ge 0$ and $l\ge 1$, we obtain
\begin{flalign}\label{79}
&\left(\int_{B_{r}}\eta^{2^{*}}\left[1+H_{\delta}(Dv_{j,\delta})^{\left(\gamma+\frac{p}{2}\right)\frac{2^{*}}{2}}\right] \ \dx\right)^{\frac{2}{2^{*}}}\le \left(\int_{B_{r}}\eta^{2^{*}}\left[1+H_{\delta}(Dv_{j,\delta})^{\left(\gamma+\frac{p}{2}\right)}\right]^{\frac{2^{*}}{2}} \ \dx\right)^{\frac{2}{2^{*}}}\nonumber \\
&\qquad \qquad \le c\int_{B_{r}}\left | \ D\left[\eta\left(1+H_{\delta}(Dv_{j,\delta})^{\frac{\gamma}{2}+\frac{p}{4}}\right)\right] \ \right |^{2} \ \dx\nonumber \\
&\qquad\qquad\le c\int_{B_{r}}\snr{D\eta}^{2}\left[1+H_{\delta}(Dv_{j,\delta})^{\gamma+\frac{p}{2}}\right] \ \dx\nonumber \\
&\qquad \qquad+c(\gamma+1)^{2}\int_{B_{r}}\eta^{2}H_{\delta}(Dv_{j,\delta})^{\gamma+\frac{p-4}{2}}\snr{D(H_{\delta}(Dv_{j,\delta}))}^{2} \ \dx\nonumber \\
&\qquad \qquad\le c(\texttt{data}_{\infty})(\gamma+1)^{2}\left(\int_{B_{r}}(\eta^{2m}+\snr{D\eta}^{2m})\left[1+H_{\delta}(Dv_{j,\delta})^{m\left(q-\frac{p}{2}+\gamma\right)}\right] \ \dx\right)^{\frac{1}{m}},
\end{flalign}
where we set 
\begin{flalign}\label{2*}
2^{*}:=\begin{cases} 
\ \frac{2n}{n-2}\quad &\mbox{if} \ \ n>2\\
\ \mbox{any number larger than} \ \frac{2sp}{3sp-2(p+sq)}\quad &\mbox{if} \ \ n=2.
\end{cases}
\end{flalign}
\subsection{Moser's iteration}
We shall use the modified Moser's iteration developed in \cite{demi}. For every integer $\kappa\ge 1$, we define by induction the exponents
\begin{flalign*}
\gamma_{1}:=0,\quad \gamma_{\kappa+1}:=\frac{1}{m}\left[\left(\gamma_{\kappa}+\frac{p}{2}\right)\frac{2^{*}}{2}-\frac{p}{2}\right],\quad \lambda_{\kappa}:=m\gamma_{\kappa}+\frac{p}{2}.
\end{flalign*}
It follows that
\begin{flalign}\label{75}
\lambda_{\kappa+1}=\left(\gamma_{\kappa}+\frac{p}{2}\right)\frac{2^{*}}{2}=:\zeta \lambda_{\kappa}+\tau,
\end{flalign}
where
\begin{flalign}\label{76}
\zeta:=\frac{2^{*}}{2m}\stackrel{s>n}{>}1\quad \mbox{and}\quad \tau:=\frac{2^{*}\lambda_{1}}{s}=\frac{2^{*}p}{2s}>0.
\end{flalign}
From \eqref{75} we have that for all integers $\kappa\ge 1$, there holds
\begin{flalign}\label{77}
\lambda_{\kappa+1}=\zeta^{\kappa}\lambda_{1}+\tau\sum_{i=0}^{\kappa-1}\zeta^{i}\quad \mbox{thus}\quad \gamma_{\kappa+1}=\frac{\lambda_{1}}{m}(\zeta^{\kappa}-1)+\frac{\tau}{m}\sum_{i=0}^{\kappa-1}\zeta^{i},
\end{flalign}
and, being $\zeta>1$, then $\lambda_{\kappa+1}>\lambda_{\kappa}$. Moreover, it is easy to see that
\begin{flalign}\label{78}
\gamma_{\kappa+1}\le \lambda_{1}\frac{\zeta^{\kappa}}{\zeta-1}\left(1+\frac{2^{*}}{s}\right)\le c(n,p,s)\zeta^{\kappa}.
\end{flalign}
From now on, all the balls considered will be concentric to $B_{r}$. We abbreviate
\begin{flalign*}
M(t)_{j,\delta}:=\nr{H_{\delta}(Dv_{j,\delta})}_{L^{\infty}(B_{t})}\quad \mbox{for all} \  t\in (0,r)
\end{flalign*}
and notice that, by \eqref{48}, $M_{j,\delta}(t)$ is bounded on any interval $[\rr_{1},\rr_{2}]$ with $0<\rr_{1}<\rr_{2}<r$. For $0<\rr\le \tau_{1} < \tau_{2}<r$, we consider a sequence $\{B_{\rr_{\kappa}}\}$ of shrinking balls, where $\rr_{\kappa }:=\tau_{1}+(\tau_{2}-\tau_{1})2^{-\kappa +1}$. Notice that $\{\varrho_{\kappa}\}$ is a decreasing sequence such that $\varrho_{1}=\tau_{2}$ and $\varrho_{\kappa}\to_{\kappa\to\infty} \tau_1$; therefore it is $\cap_{\kappa\in \N} B_{\varrho_{\kappa}} =B_{\tau_{1}}$ and $B_{\varrho_{1}}= B_{\tau_{2}}$. Accordingly, we fix corresponding cut-off functions $\eta_{\kappa}\in C^1_{c}(B_{r})$ with
\begin{flalign*}
\chi_{B_{\rr_{\kappa+1}}}\le \eta\le \chi_{B_{\rr_{\kappa}}}\quad \mbox{and}\quad \snr{D\eta_{\kappa}}\le \frac{4}{(\rr_{\kappa}-\rr_{\kappa+1})}=\frac{2^{\kappa+2}}{(\tau_{2}-\tau_{1})}.
\end{flalign*}
We fix $\eta=\eta_{\kappa}$ in \eqref{79} and rearrange it as to obtain
\begin{flalign}\label{80}
&\left(\int_{B_{\rr_{\kappa+1}}}\left[1+H_{\delta}(Dv_{j,\delta})^{\lambda_{\kappa+1}}\right] \ \dx\right)\nonumber \\
&\qquad\le c(\texttt{data}_{\infty})\left[1+M_{j,\delta}(\tau_{2})^{\frac{2^{*}\sigma}{2}}\right]\left(\frac{2^{\kappa}(\gamma_{\kappa}+1)}{\tau_{2}-\tau_{1}}\right)^{2^{*}}\left(\int_{B_{\rr_{\kappa}}}\left[1+H_{\delta}(Dv_{j,\delta})^{\lambda_{\kappa}}\right] \ \dx\right)^{\zeta},
\end{flalign}
where we set $\sigma:=q-\frac{p}{2}-\frac{p}{2m}>0$, since $q>p$ and $m>1$. For $\kappa \in \N$ we set
\begin{flalign*}
A_{\kappa}:=\left(\int_{B_{\rr_{\kappa}}}\left[1+H_{\delta}(Dv_{j,\delta})^{\lambda_{\kappa}}\right] \ \dx\right)^{\frac{1}{\lambda_{\kappa}}},
\end{flalign*}
thus \eqref{80} reads as
\begin{flalign}\label{81}
A_{\kappa+1}\le \left[1+M_{j,\delta}(\tau_{2})^{\frac{2^{*}\sigma}{2\lambda_{\kappa+1}}}\right]\left(\frac{c2^{\kappa}(\gamma_{\kappa}+1)}{\tau_{2}-\tau_{1}}\right)^{\frac{2^{*}}{\lambda_{\kappa+1}}}A_{\kappa}^{\frac{\lambda_{\kappa}\zeta}{\lambda_{\kappa+1}}},
\end{flalign}
with $c=c(\texttt{data}_{\infty})$. Iterating the inequality in \eqref{81} we obtain
\begin{flalign}\label{82}
A_{\kappa+1}\le\left[1+ M_{j,\delta}(\tau_{2})^{\frac{2^{*}\sigma}{2\lambda_{\kappa+1}}\sum_{i=0}^{\kappa-1}\zeta^{i}}\right]\prod_{i=0}^{\kappa-1}\left(\frac{c2^{\kappa-i}(\gamma_{\kappa-i}+1)}{\tau_{2}-\tau_{1}}\right)^{\frac{2^{*}\zeta^{i}}{\lambda_{\kappa+1}}}A_{1}^{\frac{\zeta^{\kappa}\lambda_{1}}{\lambda_{\kappa+1}}}
\end{flalign}
for all $\kappa\ge 1$. By \eqref{77} and simple comparison arguments
\begin{flalign*}
\frac{1}{\lambda_{\kappa+1}}\sum_{i=0}^{\kappa-1}\zeta^{i}\le \frac{1}{\lambda_{1}(\zeta-1)}\quad \mbox{and}\quad \lim_{\kappa\to \infty}\sum_{i=0}^{\kappa-1}(\kappa-i)\zeta^{i-\kappa}\le c(\zeta)<\infty,
\end{flalign*}
we get that
\begin{flalign}
\prod_{i=0}^{\kappa-1}&\left(\frac{c2^{\kappa-i}(\gamma_{\kappa-i}+1)}{\tau_{2}-\tau_{1}}\right)^{\frac{2^{*}\zeta^{i}}{\lambda_{\kappa+1}}}\stackrel{\eqref{78}}{\le}\prod_{i=0}^{\kappa-1}\left(\frac{c2^{\kappa-i}\zeta^{\kappa-i}}{\tau_{2}-\tau_{1}}\right)^{\frac{2^{*}\zeta^{i}}{\lambda_{\kappa+1}}}\nonumber \\
&\qquad \le \left(\frac{c}{\tau_{2}-\tau_{1}}\right)^{\frac{2}{p(\zeta-1)}}(2\zeta)^{\frac{2^{*}}{\lambda_{1}}\sum_{i=0}^{\kappa-1}(\kappa-i)\zeta^{i-\kappa}}\le \left(\frac{c}{\tau_{2}-\tau_{1}}\right)^{\hat{\theta}},\label{84}
\end{flalign}
with $\hat{\theta}=\hat{\theta}(n,s,p)$ and $c=c(\texttt{data}_{\infty})$, see \cite[Section 4.3]{demi} for more details. With \eqref{84} at hand we can further bound \eqref{82} to obtain
\begin{flalign}\label{83}
A_{\kappa+1}\le \left[1+M_{j,\delta}(\tau_{2})^{\frac{2^{*}\sigma}{2\lambda_{\kappa+1}}\sum_{i=0}^{\kappa-1}\zeta^{i}}\right]\left(\frac{c}{\tau_{2}-\tau_{1}}\right)^{\hat{\theta}}A_{1}^{\frac{\zeta^{\kappa}\lambda_{1}}{\lambda_{\kappa+1}}}.
\end{flalign}
Finally, notice that
\begin{flalign*}
\lim_{\kappa\to \infty}\frac{1}{\lambda_{\kappa+1}}\sum_{i=0}^{\kappa-1}\zeta^{i}=\frac{1}{(\zeta-1)\lambda_{1}+\tau}=:\theta_{1}\quad \mbox{and}\quad \lim_{\kappa\to \infty}\frac{\zeta^{\kappa}\lambda_{1}}{\lambda_{\kappa+1}}=\frac{\lambda_{1}(\zeta-1)}{\lambda_{1}(\zeta-1)+\tau}=:\theta_{2},
\end{flalign*}
so we can send $\kappa\to \infty$ in \eqref{83} and conclude with
\begin{flalign}\label{85}
M_{j,\delta}(\tau_{1})\le\left(\frac{c}{\tau_{2}-\tau_{1}}\right)^{\hat{\theta}} \left[1+M_{j,\delta}(\tau_{2})^{\frac{2^{*}\sigma}{2}\theta_{1}}\right]A_{1}^{\theta_{2}}
\end{flalign}
Since 
\begin{flalign*}
\frac{2^{*}\sigma}{2}\theta_{1}<1\Leftrightarrow\frac{q}{p}<1+\frac{1}{2}-\frac{1}{s}-\frac{1}{2^{*}},
\end{flalign*}
which is the case by \eqref{pqh} and \eqref{2*}. Hence, we can apply Young inequality to \eqref{85} with conjugate exponents $\theta_{3}:=\frac{2}{2^{*}\sigma\theta_{1}}$ and $\theta_{4}:=\frac{2}{2-2^{*}\sigma \theta_{1}}$ to get
\begin{flalign}\label{86}
M_{j,\delta}(\tau_{1})\le& \frac{1}{2}M_{j,\delta}(\tau_{2})+\left(\frac{c}{\tau_{2}-\tau_{1}}\right)^{\hat{\theta}}A_{1}^{\theta_{2}}+\left(\frac{c}{\tau_{2}-\tau_{1}}\right)^{\hat{\theta}\theta_{4}}A_{1}^{\theta_{2}\theta_{4}}\nonumber \\
\le &\frac{1}{2}M_{j,\delta}(\tau_{2})+\left(\frac{c(\texttt{data}_{\infty})}{\tau_{2}-\tau_{1}}\right)^{\tilde{\theta}}\left[1+A_{1}\right]^{\bar{\theta}},
\end{flalign}
where we set $\tilde{\theta}:=\hat{\theta}\theta_{4}$ and $\bar{\theta}:=\theta_{2}\theta_{4}$, thus $\tilde{\theta}=\tilde{\theta}(n,p,q,s)$ and $\bar{\theta}=\bar{\theta}(n,p,q,s)$. Finally, Lemma \ref{l5} and \eqref{86} render that
\begin{flalign}\label{87}
\nr{H_{\delta}(Dv_{j,\delta})}_{L^{\infty}(B_{\rr})}\le \left(\frac{c}{r-\rr}\right)^{\tilde{\theta}}\left[1+\left(\int_{B_{r}}H_{\delta}(Dv_{j,\delta})^{\frac{p}{2}} \ \dx\right)\right]^{\theta},
\end{flalign}
for $c=c(\texttt{data}_{\infty})$, $\tilde{\theta}=\tilde{\theta}(n,p,q,s)$ and $\theta=\theta(n,p,q,s)$.

\subsection{Convergence} Looking at the very definition of problem \eqref{opjd}, we fix an arbitrary $j\in \N$ and using $\eqref{assfjd}_{1}$, $\eqref{44}_{2,3}$ and \eqref{43} we get
\begin{flalign}\label{88}
\frac{\varepsilon_{j}}{q}\int_{B_{r}}\snr{Dv_{j,\delta}}^{q} \ \dx \le&\frac{\varepsilon_{j}}{q}\int_{B_{r}}H_{\delta}(Dv_{j,\delta})^{\frac{q}{2}} \ \dx \le \mathcal{F}_{j,\delta}(v_{j,\delta},B_{r})\nonumber \\
\le& \left[\int_{B_{r}}F_{\delta}(x,D\tilde{v}_{j}) \ \dx+\frac{\varepsilon_{j}}{q}\int_{B_{r}}(\mu_{\delta}^{2}+\snr{D\tilde{v}_{j}}^{2})^{\frac{q}{2}} \ \dx\right]\nonumber \\
\le&\int_{B_{r}}F(x,D\tilde{v}_{j})\ \dx +\left[\int_{B_{r}}\left[F_{\delta}(x,D\tilde{v}_{j})-F(x,D\tilde{v}_{j})\right] \ \dx\right]+o(j)\nonumber \\
=&\int_{B_{r}}F(x,D\tilde{v})\ \dx +\int_{B_{r}}\left[F_{\delta}(x,D\tilde{v}_{j})-F(x,D\tilde{v}_{j})\right] \ \dx+o(j).
\end{flalign}
Since $j\in \N$ is fixed and $D\tilde{v}_{j}\in W^{1,q}(B_{r})$, by \eqref{39} we have
\begin{flalign}\label{89}
\int_{B_{r}}\left[F_{\delta}(x,D\tilde{v}_{j})-F(x,D\tilde{v}_{j})\right] \ \dx=o_{j}(\delta)\to_{\delta\to 0}0,
\end{flalign}
therefore the sequence $\{Dv_{j,\delta}\}_{\delta>0}$ is bounded in $L^{q}(B_{r})$ uniformly in $\delta>0$. Hence, up to extract a (non-relabelled) subsequence (depending on the chosen index $j\in \N$), we find that
\begin{flalign}\label{93}
v_{j,\delta}\rightharpoonup_{\delta\to 0}v_{j} \ \ \mbox{in} \ W^{1,q}(B_{r}),\quad v_{j}\in \tilde{v}_{j}+W^{1,q}_{0}(B_{r}),\quad v_{j}\ge \psi \ \ \mbox{a.e. in} \ B_{r}.
\end{flalign}
From \eqref{85}, \eqref{88}, \eqref{89} and $\eqref{assfjd}_{1}$ it follows that
\begin{flalign}\label{90}
\nr{Dv_{j,\delta}}_{L^{\infty}(B_{\rr})}\le \left(\frac{c}{r-\rr}\right)^{\tilde{\theta}}\left[1+\mathcal{F}(v,B_{r})+o_{j}(\delta)+o(j)\right]^{\theta},
\end{flalign}
with $c=c(\texttt{data}_{\infty})$, $\tilde{\theta}=\tilde{\theta}(n,p,q,s)$ and $\theta=\theta(n,p,q,s)$. This implies that, again up to subsequences, $Dv_{j,\delta}\rightharpoonup^{*}_{\delta\to 0}Dv_{j}$ in $W^{1,\infty}(B_{\rr},\mathbb{R}^{n})$, so by weak$^{*}$-lower semincontinuity we can send $\delta\to 0$ in \eqref{90} to end up with
\begin{flalign}\label{91}
\nr{Dv_{j}}_{L^{\infty}(B_{\rr})}\le \left(\frac{c}{r-\rr}\right)^{\tilde{\theta}}\left[1+\mathcal{F}(v,B_{r})+o(j)\right]^{\theta},
\end{flalign}
for $c=c(\texttt{data}_{\infty})$, $\tilde{\theta}=\tilde{\theta}(n,p,q,s)$ and $\theta=\theta(n,p,q,s)$. Notice that \eqref{91} actually holds for all concentric balls $B_{\rr}\subset B_{r}$ with $\rr\in (0,r)$. Now, by \eqref{39} and \eqref{90} we have
\begin{flalign*}
\lim_{\delta\to 0}\int_{B_{\rr}}\left[F_{\delta}(x,Dv_{j,\delta})-F(x,Dv_{j,\delta})\right] \ \dx =0
\end{flalign*}
and, by weak lower semicontinuity there holds that
\begin{flalign*}
\mathcal{F}(v_{j},B_{\rr})\le \liminf_{\delta\to 0}\mathcal{F}(v_{j,\delta},B_{\rr}).
\end{flalign*}
Merging all the above informations we obtain
\begin{flalign*}
\mathcal{F}(v_{j},B_{\rr})\le& \liminf_{\delta\to 0}\int_{B_{\rr}}F_{\delta}(x,Dv_{j,\delta}) \ \dx\le \limsup_{\delta\to 0}\mathcal{F}_{j,\delta}(v_{j,\delta},B_{r})\le \mathcal{F}(v,B_{r})+o(j),
\end{flalign*}
where for the last inequality we also used \eqref{88} and \eqref{89}. Letting $\rr\to r$ in the previous display we see that
\begin{flalign}\label{92}
\mathcal{F}(v_{j},B_{r})\le \mathcal{F}(v,B_{r})+o(j).
\end{flalign}
By $\eqref{assfh}_{1}$, \eqref{92} and the arbitrariety of $\rr\in (0,r)$, we deduce that the sequence $\{Dv_{j}\}_{j\in \N}$ is uniformly bounded in $L^{p}(B_{r},\mathbb{R}^{n})$, therefore, recalling also $\eqref{93}_{3}$ and $\eqref{44}_{1}$, we get that
\begin{flalign}\label{94}
v_{j}\rightharpoonup_{j\to \infty}\tilde{v} \ \ \mbox{in} \ W^{1,p}(B_{r}), \quad \tilde{v}\in v+W^{1,p}_{0}(B_{r}),\quad \tilde{v}\ge \psi \ \ \mbox{a.e. in} \ B_{r},
\end{flalign}
thus $\tilde{v}\in \mathcal{K}_{\psi,v}(B_{r})$. Moreover, combining \eqref{91} and $\eqref{94}_{1}$ we also obtain that
\begin{flalign}\label{95}
v_{j}\rightharpoonup_{j\to \infty}^{*}v \quad \mbox{in} \ \ W^{1,\infty}(B_{\rr}), 
\end{flalign}
for all balls $B_{\rr}$ concentric to $B_{r}$ with $\rr\in (0,r)$. Weak$^{*}$-lower semicontinuity, \eqref{95} and \eqref{91} render that
\begin{flalign}\label{96}
\nr{D\tilde{v}}_{L^{\infty}(B_{\rr})}\le \left(\frac{c}{r-\rr}\right)^{\tilde{\theta}}\left[1+\mathcal{F}(v,B_{r})\right]^{\theta},
\end{flalign}
with $c=c(\texttt{data}_{\infty})$, $\tilde{\theta}=\tilde{\theta}(n,p,q,s)$ and $\theta=\theta(n,p,q,s)$. Now we can exploit $\eqref{94}_{1}$ and weak-lower semicontinuity to pass to the limit in \eqref{92} and obtain
\begin{flalign}\label{97}
\mathcal{F}(\tilde{v},B_{r})\le \liminf_{j\to \infty}\mathcal{F}(v_{j},B_{r})\le \limsup_{j\to \infty}\mathcal{F}(v_{j},B_{r})\le \mathcal{F}(v,B_{r}).
\end{flalign}
Combining \eqref{97}, $\eqref{94}_{3}$, the minimality of $v$ in class $\mathcal{K}_{\psi,v}(\Omega)$ and \eqref{98} we can conclude that $\tilde{v}=v$ a.e. on $B_{r}$ thus estimate \eqref{96} holds for $v$ as well. Finally, via a standard covering argument we get that $v\in W^{1,\infty}_{loc}(\Omega)$ and the proof is complete.

\section{Weak differentiability for obstacle problems with standard $q$-growth}\label{diff}
In this section we prove a higher regularity result for solutions of non-autonomous obstacle problems with standard polynomial growth. Precisely, we shall consider an integrand $F\colon \Omega\times \mathbb{R}^{n}\to \mathbb{R}$ satisfying
\begin{flalign}\label{asss}
\begin{cases}
\ z\mapsto F(\cdot,z)\in C^{2}_{loc}(\mathbb{R}^{n})\\
\ x\mapsto F(x,z)\in C^{\infty}_{loc}(\Omega) \ \ \mbox{for all} \ z\in \mathbb{R}^{n}\\
\ \tilde{\nu}\left[\htt(z)^{\frac{p}{2}}+\htt(z)^{\frac{q}{2}}\right]\le F(x,z)\le \tilde{L}\left[\htt(z)^{\frac{p}{2}}+\htt(z)^{\frac{q}{2}}\right]\\
\ \tilde{\nu}\left[\htt(z)^{\frac{p-2}{2}}+\htt(z)^{\frac{q-2}{2}}\right]\snr{\xi}^{2}\le \partial^{2}_{z}F(x,z)\xi\cdot \xi\\
\ \snr{\partial^{2}_{z}F(x,z)}\le \tilde{L}\left[\htt(z)^{\frac{p-2}{2}}+\htt(z)^{\frac{q-2}{2}}\right]\\
\ \snr{\partial^{2}_{x,z}F(x,z)}\le \tilde{L}\left[\htt(z)^{\frac{p-1}{2}}+\htt(z)^{\frac{q-1}{2}}\right],
\end{cases}
\end{flalign}
for all $x\in \Omega$ and $z,\xi \in \mathbb{R}^{n}$. Here, $0<\tilde{\nu}\le \tilde{L}$ are absolute constants and we set $\htt(z):=(\tilde{\mu}^{2}+\snr{z}^{2})$ with $\tilde{\mu}\in (0,1)$. For the obstacle function $\psi\colon \Omega\to \mathbb{R}$, we shall retain \eqref{obsh}. We study regularity for local minimizers of the variational integral with obstacle constraint
\begin{flalign}\label{ops}
\mathcal{K}_{\psi}^{*}(\Omega)\ni w\mapsto \min \int_{\Omega}F(x,Dw) \ \dx,
\end{flalign}
where this time
\begin{flalign*}
\mathcal{K}_{\psi}^{*}(\Omega):=\left\{w\in W^{1,q}(\Omega)\colon w\ge \psi \ \mbox{a.e. in}\ \Omega\right\}.
\end{flalign*}
Of course we are supposing that
\begin{flalign}\label{105}
\mathcal{K}_{\psi}^{*}(\Omega)\quad \mbox{is non-empty}.
\end{flalign}
Our main result in this perspective is the following
\begin{proposition}\label{regst}
Let $v\in \mathcal{K}_{\psi}^{*}(\Omega)$ be a solution of problem \eqref{ops} under assumptions \eqref{asss}, \eqref{obsh} and \eqref{105}. Then
\begin{flalign*}
v\in C^{1,\beta_{0}}_{loc}(\Omega) 
\end{flalign*}
for some $\beta_{0}=\beta_{0}(n,\tilde{\nu},\tilde{L},p,q)\in (0,1)$. Moreover, there holds that
\begin{flalign}\label{51}
v\in W^{2,2}_{loc}(\Omega)\quad \mbox{and}\quad \partial_{z}F(x,Dv)\in W^{1,2}_{loc}(\Omega,\mathbb{R}^{n}).
\end{flalign}
\end{proposition}
\begin{proof}
First notice that, by $\eqref{asss}_{1,3}$ we can compute the variational inequality associated to problem \eqref{ops}: we have that 
\begin{flalign}\label{52}
\int_{\Omega}\partial_{z}F(x,Dv)\cdot (Dw-Dv) \ \dx \ge 0\quad \mbox{for all} \ \ w\in \mathcal{K}_{\psi}^{*}(\Omega).
\end{flalign}
The local $C^{1,\beta_{0}}$-regularity follows from the results in \cite{ch,chle,muzi}, but for our ends $v\in W^{1,\infty}_{loc}(\Omega)$ will be enough. To prove the weak higher differentiability of $Dv$, as in Section \ref{fra}, we fix a ball $B_{r}\Subset \Omega$, $r\in (0,1]$, pick a cut-off function $\eta\in C^{1}_{c}(B_{r})$ so that
\begin{flalign*}
\chi_{B_{r/4}}\le \eta \le \chi_{B_{r/2}}\quad \mbox{and}\quad \snr{D\eta}\le \frac{4}{r},
\end{flalign*}
a vector $h\in \mathbb{R}^{n}$ with $\snr{h}\le \frac{1}{1000}\min\left\{\frac{r}{16},\dist(\partial B_{r},\partial \Omega)\right\}$ and test \eqref{52} against the map $w:=v+\tau_{-h}(\eta^{2}\tau_{h}(v-\psi))\in \mathcal{K}_{v,\psi}^{*}(B_{r})$. We obtain
\begin{flalign*}
0\le& -\int_{B_{r}}\tau_{h}(\partial_{z}F(x,Dv))\cdot D(\eta^{2}\tau_{h}(v-\psi)) \ \dx \nonumber \\
&=\mbox{(I)}+\mbox{(II)}+\mbox{(III)}+\mbox{(IV)}+\mbox{(V)}+\mbox{(VI)}.
\end{flalign*}
The decomposition into terms (I)-(VI) is the same appearing in Section \ref{hi}, but the resulting estimates will be slightly different from what we did before, owing to the higher regularity we are assuming now for both integrand, obstacle and solution. For simplicity we shall separate the three cases $q>p\ge 2$, $1<p<2\le q$ and $1<p<q\le 2$.\\\\
\emph{Case 1: $2\le p<q$}. By $\eqref{asss}_{4}$ and Lemma \ref{l1} we have
\begin{flalign*}
\mbox{(I)}:=&-\int_{B_{r}}\eta^{2}\left[\partial_{z} F(x+h,Dv(x+h))-\partial_{z} F(x+h,Dv(x))\right]\cdot \tau_{h}(Dv) \ \dx \nonumber \\
\le &-c\int_{B_{r}}\eta^{2}\left[\snr{\tau_{h}(V_{\tilde{\mu},p}(Dv))}^{2}+\snr{\tau_{h}(V_{\tilde{\mu},q}(Dv))}^{2}\right] \ \dx,
\end{flalign*}
for $c=c(n,\tilde{\nu},p,q)$. Using $\eqref{asss}_{4}$, the mean value theorem, Lemmas \ref{l1}, \ref{l2}, \ref{l3} and \ref{l6}, \eqref{obsh}, H\"older and Young inequalities we get
\begin{flalign*}
\snr{\mbox{(II)}}:=&\int_{B_{r}}\eta^{2}\left[\partial_{z} F(x+h,Dv(x+h))-\partial_{z} F(x+h,Dv(x))\right]\cdot \tau_{h}(D\psi) \ \dx\nonumber \\
\le &c\int_{B_{r}}\eta^{2}\left[\htt(Dv(x))^{\frac{p-2}{2}}+\htt(Dv(x+h))^{\frac{p-2}{2}}\right]\snr{\tau_{h}(Dv)}\snr{\tau_{h}(D\psi)} \ \dx \nonumber \\
&+c\int_{B_{r}}\eta^{2}\left[\htt(Dv(x))^{\frac{q-2}{2}}+\htt(Dv(x+h))^{\frac{q-2}{2}}\right]\snr{\tau_{h}(Dv)}\snr{\tau_{h}(D\psi)} \ \dx \nonumber \\
\le &\sigma \left(\int_{B_{r}}\eta^{2}\left[\snr{\tau_{h}(V_{\tilde{\mu},p}(Dv))}^{2}+\snr{\tau_{h}(V_{\tilde{\mu},q}(Dv))}^{2}\right] \ \dx\right) \ \dx \nonumber \\
&+c\left(\int_{B_{3r/4}}\htt(Dv)^{\frac{p}{2}} \ \dx \right)^{\frac{p-2}{p}}\left(\int_{B_{3r/4}}\snr{\tau_{h}(D\psi)}^{p} \ \dx\right)^{\frac{2}{p}}\nonumber \\
&+c\left(\int_{B_{3r/4}}\htt(Dv)^{\frac{q}{2}} \ \dx \right)^{\frac{q-2}{q}}\left(\int_{B_{3r/4}}\snr{\tau_{h}(D\psi)}^{q} \ \dx\right)^{\frac{2}{q}}\nonumber \\
\le &\sigma \left(\int_{B_{r}}\eta^{2}\left[\snr{\tau_{h}(V_{\tilde{\mu},p}(Dv))}^{2}+\snr{\tau_{h}(V_{\tilde{\mu},q}(Dv))}^{2}\right] \ \dx\right) \ \dx+c\snr{h}^{2},
\end{flalign*}
for $c=c(n,\tilde{\nu},\tilde{L},p,q,\sigma,\nr{Dv}_{L^{\infty}(B_{3r/4})},\nr{\psi}_{W^{2,\infty}(\Omega)})$. From $\eqref{asss}_{6}$, Lemmas \ref{l3}, \ref{l1} and \ref{l6}, $\eqref{obsh}$, H\"older and Young inequalities we obtain
\begin{flalign*}
\snr{\mbox{(III)}}:=&2\left| \ \int_{B_{r}}\eta\left[\partial_{z} F(x+h,Dv(x+h))-\partial_{z} F(x+h,Dv(x))\right]\cdot \tau_{h}(v-\psi)D\eta \ \dx\ \right |\nonumber \\
\le &c\int_{B_{r}}\eta\left[\htt(Dv)^{\frac{p-2}{2}}+\htt(Dv)^{\frac{q-2}{2}}\right]\snr{\tau_{h}(Dv)}\snr{\tau_{h}(v-\psi)}\snr{D\eta} \ \dx \nonumber \\
\le &\frac{\sigma}{c}\int_{B_{r}}\eta^{2}\left[\snr{\tau_{h}(V_{\tilde{\mu},p}(Dv))}^{2}+\snr{\tau_{h}(V_{\tilde{\mu},q}(Dv))}^{2}\right] \ \dx\nonumber \\
&+\sigma^{-1}\frac{c}{r^{2}}\int_{B_{r}}\left[\htt(Dv)^{\frac{p-2}{2}}+\htt(Dv)^{\frac{q-2}{2}}\right]\snr{\tau_{h}(v-\psi)}^{2} \ \dx\nonumber \\
\le &\frac{\sigma}{c}\int_{B_{r}}\eta^{2}\left[\snr{\tau_{h}(V_{\tilde{\mu},p}(Dv))}^{2}+\snr{\tau_{h}(V_{\tilde{\mu},q}(Dv))}^{2}\right] \ \dx\nonumber \\
&+\sigma^{-1}\frac{c}{r^{2}}\left(\int_{B_{3r/4}}H_{\delta}(Dv)^{\frac{p}{2}} \ \dx\right)^{\frac{p-2}{p}}\left(\int_{B_{r/2}}\snr{\tau_{h}(v-\psi)}^{p} \ \dx\right)^{\frac{2}{p}}\nonumber \\
&+\sigma^{-1}\frac{c}{r^{2}}\left(\int_{B_{3r/4}}H_{\delta}(Dv)^{\frac{q}{2}} \ \dx\right)^{\frac{q-2}{q}}\left(\int_{B_{r/2}}\snr{\tau_{h}(v-\psi)}^{q} \ \dx\right)^{\frac{2}{q}}\nonumber \\
\le &\frac{\sigma}{c}\int_{B_{r}}\eta^{2}\left[\snr{\tau_{h}(V_{\tilde{\mu},p}(Dv))}^{2}+\snr{\tau_{h}(V_{\tilde{\mu},q}(Dv))}^{2}\right] \ \dx\nonumber \\
&+\sigma^{-1}\frac{c\snr{h}^{2}}{r^{2}}\left(\int_{B_{3r/4}}H_{\delta}(Dv)^{\frac{p}{2}} \ \dx\right)^{\frac{p-2}{p}}\left(\int_{B_{3r/4}}\snr{Dv-D\psi}^{p} \ \dx\right)^{\frac{2}{p}}\nonumber \\
&+\sigma^{-1}\frac{c\snr{h}^{2}}{r^{2}}\left(\int_{B_{3r/4}}H_{\delta}(Dv)^{\frac{q}{2}} \ \dx\right)^{\frac{q-2}{q}}\left(\int_{B_{3r/4}}\snr{Dv-D\psi}^{q} \ \dx\right)^{\frac{2}{q}}\nonumber \\
\le &\frac{\sigma}{c}\int_{B_{r}}\eta^{2}\left[\snr{\tau_{h}(V_{\tilde{\mu},p}(Dv))}^{2}+\snr{\tau_{h}(V_{\tilde{\mu},q}(Dv))}^{2}\right] \ \dx+\frac{c\snr{h}^{2}}{r^{2}},
\end{flalign*}
where $c=c(n,\tilde{\nu},\tilde{L},p,q,\nr{Dv}_{L^{\infty}(B_{3r/4})},\nr{D\psi}_{L^{\infty}(\Omega)})$. By $\eqref{assfjd}_{5}$, Lemmas \ref{l1}, \ref{l2} and \ref{l3}, $\eqref{obsh}$, H\"older and Young inequalities we have
\begin{flalign*}
\snr{\mbox{(IV)}}&+\snr{\mbox{(V)}}:=\left | \ \int_{B_{r}}\eta^{2}\left[\partial_{z}F(x+h,Dv(x))-\partial_{z}F(x,Dv(x))\right]\cdot \tau_{h}(Dv) \ \dx  \ \right |\nonumber \\
&+\left | \ \int_{B_{r}}\eta^{2}\left[\partial_{z}F(x+h,Dv(x))-\partial_{z}F(x,Dv(x))\right]\cdot \tau_{h}(D\psi) \ \dx  \ \right |\nonumber \\
\le &c\snr{h}\int_{B_{r}}\eta^{2}\left[\htt(Dv)^{\frac{p-1}{2}}+\htt(Dv)^{\frac{q-1}{2}}\right]\snr{\tau_{h}(Dv)} \ \dx\nonumber \\
&+c\snr{h}\left(\int_{B_{r}}\eta^{2}\htt(Dv)^{\frac{p}{2}} \ \dx \right)^{\frac{p-1}{p}}\left(\int_{B_{r/2}}\snr{\tau_{h}(D\psi)}^{p} \ \dx \right)^{\frac{1}{p}}\nonumber \\
&+c\snr{h}\left(\int_{B_{r}}\eta^{2}\htt(Dv)^{\frac{q}{2}} \ \dx \right)^{\frac{q-1}{q}}\left(\int_{B_{r/2}}\snr{\tau_{h}(D\psi)}^{q} \ \dx \right)^{\frac{1}{q}}\nonumber \\
\le &\sigma \int_{B_{r}}\eta^{2}(\tilde{\mu}^{2}+\snr{Dv(x+h)}^{2}+\snr{Dv(x)}^{2})^{\frac{p-2}{2}} \snr{\tau_{h}(Dv)}^{2}\ \dx\nonumber \\
&+\sigma\int_{B_{r}}\eta^{2}(\tilde{\mu}^{2}+\snr{Dv(x+h)}^{2}+\snr{Dv(x)}^{2})^{\frac{q-2}{2}}\snr{\tau_{h}(Dv)}^{2} \ \dx \nonumber \\
&+c\snr{h}^{2}\int_{B_{3r/4}}\left[\htt(Dv)^{\frac{p}{2}}+\htt(Dv)^{\frac{q}{2}}\right] \ \dx +c\snr{h}^{2}\nonumber\\
\le&c\sigma \int_{B_{r}}\eta^{2}\left[\snr{\tau_{h}(V_{\tilde{\mu},p}(Dv))}^{2}+\snr{\tau_{h}(V_{\tilde{\mu},q}(Dv))}^{2}\right] \ \dx+c\snr{h}^{2},
\end{flalign*}
with $c=c(n,\tilde{\nu},\tilde{L},p,q,\sigma,\nr{Dv}_{L^{\infty}(B_{3r/4})},\nr{\psi}_{W^{2,\infty}(\Omega)})$. Finally, exploiting $\eqref{asss}_{6}$, Lemmas \ref{l2} and \ref{l3}, $\eqref{obsh}$, H\"older and Young inequality we end up with
\begin{flalign*}
\snr{(\mbox{VI})}:=&2\left| \ \int_{B_{r}}\eta \left[\partial_{z}F(x+h,Dv)-\partial_{z}F(x,Dv)\right]\cdot (\tau_{h}(v-\psi)D\eta) \ \dx \ \right |\nonumber \\
\le& \frac{c\snr{h}}{r}\int_{B_{r}}\left[\htt(Dv)^{\frac{p-1}{2}}+\htt(Dv)^{\frac{q-1}{2}}\right]\snr{\tau_{h}(v-\psi)} \ \dx \nonumber \\
\le &\frac{c\snr{h}}{r}\left(\int_{B_{r/2}}\htt(Dv)^{\frac{p}{2}} \ \dx\right)^{\frac{p-1}{p}}\left(\int_{B_{r/2}}\snr{\tau_{h}(v-\psi)}^{p} \ \dx \right)^{\frac{1}{p}}\nonumber \\
&+\frac{c\snr{h}}{r}\left(\int_{B_{r/2}}\htt(Dv)^{\frac{q}{2}} \ \dx\right)^{\frac{q-1}{q}}\left(\int_{B_{r/2}}\snr{\tau_{h}(v-\psi)}^{q} \ \dx \right)^{\frac{1}{q}}\nonumber \\
\le &\frac{c\snr{h}^{2}}{r}\left(\int_{B_{3r/4}}\htt(Dv)^{\frac{p}{2}} \ \dx\right)^{\frac{p-1}{p}}\left(\int_{B_{3r/4}}\snr{Dv-D\psi}^{p} \ \dx \right)^{\frac{1}{p}}\nonumber \\
&+\frac{c\snr{h}^{2}}{r}\left(\int_{B_{3r/4}}\htt(Dv)^{\frac{q}{2}} \ \dx\right)^{\frac{q-1}{q}}\left(\int_{B_{3r/4}}\snr{Dv-D\psi}^{q} \ \dx \right)^{\frac{1}{q}}\le \frac{c\snr{h}^{2}}{r},
\end{flalign*}
for $c=c(n,\tilde{\nu},\tilde{L},p,q,\nr{Dv}_{L^{\infty}(B_{3r/4})},\nr{\psi}_{W^{2,\infty}(\Omega)})$.\\\\
\emph{Case 2: $1<p<2\le q$.} Only terms (II)-(III) need a different treatment. By $\eqref{asss}_{4}$, the mean value theorem, Lemmas \ref{l1}, \ref{l2}, \ref{l3} and \ref{l6}, \eqref{obsh}, H\"older and Young inequalities we have
\begin{flalign*}
\snr{(\mbox{II})}\le &\sigma \left(\int_{B_{r}}\eta^{2}\left[\snr{\tau_{h}(V_{\tilde{\mu},p}(Dv))}^{2}+\snr{\tau_{h}(V_{\tilde{\mu},q}(Dv))}^{2}\right] \ \dx\right) \nonumber \\
&+c\int_{B_{r}}\eta^{2}\left[\htt(Dv(x))^{\frac{p-2}{2}}+\htt(Dv(x+h))^{\frac{p-2}{2}}\right]\snr{\tau_{h}(D\psi)}^{2} \ \dx\nonumber \\
&+c\int_{B_{r}}\eta^{2}\left[\htt(Dv(x))^{\frac{q-2}{2}}+\htt(Dv(x+h))^{\frac{q-2}{2}}\right]\snr{\tau_{h}(D\psi)}^{2} \ \dx\nonumber \\
\le &\sigma \left(\int_{B_{r}}\eta^{2}\left[\snr{\tau_{h}(V_{\tilde{\mu},p}(Dv))}^{2}+\snr{\tau_{h}(V_{\tilde{\mu},q}(Dv))}^{2}\right] \ \dx\right) \ \dx+c\tilde{\mu}^{p-2}\int_{B_{r/2}}\snr{\tau_{h}D\psi}^{2} \ \dx\nonumber \\
&+c\left(\int_{B_{3r/4}}\htt(Dv)^{\frac{q}{2}} \ \dx \right)^{\frac{q-2}{q}}\left(\int_{B_{3r/4}}\snr{\tau_{h}(D\psi)}^{q} \ \dx\right)^{\frac{2}{q}}\nonumber \\
\le &\sigma \left(\int_{B_{r}}\eta^{2}\left[\snr{\tau_{h}(V_{\tilde{\mu},p}(Dv))}^{2}+\snr{\tau_{h}(V_{\tilde{\mu},q}(Dv))}^{2}\right] \ \dx\right)+c\snr{h}^{2},
\end{flalign*}
where $c=c(n,\tilde{\nu},\tilde{L},\tilde{\mu},p,q,\sigma,\nr{Dv}_{L^{\infty}(B_{3r/4})},\nr{\psi}_{W^{2,\infty}(\Omega)})$. Using the mean value theorem, $\eqref{asss}_{4}$, Lemma \ref{l3}, H\"older and Young inequalities we get
\begin{flalign*}
\snr{\mbox{(III)}}\le& \frac{\sigma}{c} \left(\int_{B_{r}}\eta^{2}\left[\snr{\tau_{h}(V_{\tilde{\mu},p}(Dv))}^{2}+\snr{\tau_{h}(V_{\tilde{\mu},q}(Dv))}^{2}\right] \ \dx\right)\nonumber \\
&+\sigma^{-1}\frac{c}{r^{2}}\int_{B_{r/2}}\left[\htt(Dv)^{\frac{p-2}{2}}+\htt(Dv)^{\frac{q-2}{2}}\right]\snr{\tau_{h}(v-\psi)}^{2} \ \dx \nonumber \\
\le &\frac{\sigma}{c} \left(\int_{B_{r}}\eta^{2}\left[\snr{\tau_{h}(V_{\tilde{\mu},p}(Dv))}^{2}+\snr{\tau_{h}(V_{\tilde{\mu},q}(Dv))}^{2}\right] \ \dx\right)\nonumber \\
&+\sigma^{-1}\frac{c\snr{h}^{2}}{r^{2}}\tilde{\mu}^{p-2}\int_{B_{3r/4}}\snr{Dv-D\psi}^{2} \ \dx\nonumber \\
&+\sigma^{-1}\frac{c\snr{h}^{2}}{\sigma}\left(\int_{B_{3r/4}}\htt(Dv)^{\frac{q}{2}} \ \dx\right)^{\frac{q-2}{2}}\left(\int_{B_{3r/4}}\snr{Dv-D\psi}^{q} \ \dx\right)^{\frac{2}{q}}\nonumber \\
\le &\frac{\sigma}{c} \left(\int_{B_{r}}\eta^{2}\left[\snr{\tau_{h}(V_{\tilde{\mu},p}(Dv))}^{2}+\snr{\tau_{h}(V_{\tilde{\mu},q}(Dv))}^{2}\right] \ \dx\right)+\frac{c\snr{h}^{2}}{r^{2}},
\end{flalign*}
with $c=c(n,\tilde{\nu},\tilde{L},\tilde{\mu},p,q,\sigma,\nr{Dv}_{L^{\infty}(B_{3r/4})},\nr{D\psi}_{L^{\infty}(\Omega)})$.\\\\
\emph{Case 3: $1<p<q\le 2$.} As for the previous case we bound
\begin{flalign*}
\snr{(\mbox{II})}\le &\sigma \left(\int_{B_{r}}\eta^{2}\left[\snr{\tau_{h}(V_{\tilde{\mu},p}(Dv))}^{2}+\snr{\tau_{h}(V_{\tilde{\mu},q}(Dv))}^{2}\right] \ \dx\right) \ \dx \nonumber \\
&+c\int_{B_{r}}\eta^{2}\left[\htt(Dv(x))^{\frac{p-2}{2}}+\htt(Dv(x+h))^{\frac{p-2}{2}}\right]\snr{\tau_{h}(D\psi)}^{2} \ \dx\nonumber \\
&+c\int_{B_{r}}\eta^{2}\left[\htt(Dv(x))^{\frac{q-2}{2}}+\htt(Dv(x+h))^{\frac{q-2}{2}}\right]\snr{\tau_{h}(D\psi)}^{2} \ \dx\nonumber \\
\le &\sigma \left(\int_{B_{r}}\eta^{2}\left[\snr{\tau_{h}(V_{\tilde{\mu},p}(Dv))}^{2}+\snr{\tau_{h}(V_{\tilde{\mu},q}(Dv))}^{2}\right] \ \dx\right) \ \dx \nonumber \\
&+c(\tilde{\mu}^{p-2}+\tilde{\mu}^{q-2})\int_{B_{r/2}}\snr{D\psi}^{2} \ \dx \nonumber \\
\le &\sigma \left(\int_{B_{r}}\eta^{2}\left[\snr{\tau_{h}(V_{\tilde{\mu},p}(Dv))}^{2}+\snr{\tau_{h}(V_{\tilde{\mu},q}(Dv))}^{2}\right] \ \dx\right) \ \dx +c\snr{h}^{2},
\end{flalign*}
and
\begin{flalign*}
\snr{\mbox{(III)}}\le &\sigma \left(\int_{B_{r}}\eta^{2}\left[\snr{\tau_{h}(V_{\tilde{\mu},p}(Dv))}^{2}+\snr{\tau_{h}(V_{\tilde{\mu},q}(Dv))}^{2}\right] \ \dx\right) \ \dx\nonumber \\
&+\sigma^{-1}\frac{c}{r^{2}}\left(\tilde{\mu}^{p-2}+\tilde{\mu}^{q-2}\right)\int_{B_{r/2}}\snr{\tau_{h}(v-\psi)}^{2} \ \dx \nonumber \\
\le&\sigma \left(\int_{B_{r}}\eta^{2}\left[\snr{\tau_{h}(V_{\tilde{\mu},p}(Dv))}^{2}+\snr{\tau_{h}(V_{\tilde{\mu},q}(Dv))}^{2}\right] \ \dx\right) \ \dx\nonumber \\
&+\sigma^{-1}\frac{c\snr{h}^{2}}{r^{2}}\left(\tilde{\mu}^{p-2}+\tilde{\mu}^{q-2}\right)\int_{B_{3r/4}}\snr{(Dv-D\psi)}^{2} \ \dx\nonumber \\
\le &\sigma \left(\int_{B_{r}}\eta^{2}\left[\snr{\tau_{h}(V_{\tilde{\mu},p}(Dv))}^{2}+\snr{\tau_{h}(V_{\tilde{\mu},q}(Dv))}^{2}\right] \ \dx\right) \ \dx+\sigma^{-1}\frac{c\snr{h}^{2}}{r^{2}}.
\end{flalign*}
In both the previous displays, $c=c(n,\tilde{\nu},\tilde{L},\tilde{\mu},p,q,\sigma,\nr{\psi}_{W^{2,\infty}(\Omega)})$.\\\\
Merging all the previous estimates and choosing $\sigma>0$ sufficiently small, we obtain
\begin{flalign}\label{47}
\int_{B_{r}}\eta^{2}\left[\snr{\tau_{h}(V_{\tilde{\mu},p}(Dv))}^{2}+\snr{\tau_{h}(V_{\tilde{\mu},q}(Dv))}^{2}\right] \ \dx\le \frac{c}{r^{2}}\snr{h}^{2},
\end{flalign}
for $c=c(n,\tilde{\nu},\tilde{L},\tilde{\mu},p,q,\nr{Dv}_{L^{\infty}(B_{3r/4})},\nr{\psi}_{W^{2,\infty}(\Omega)})$. Combining \eqref{47} with Lemma \ref{l8} and recalling the specifics of the cut-off $\eta$, we obtain, that $V_{\tilde{\mu},q}(Dv)\in W^{1,2}(B_{r/4},\mathbb{R}^{n})$ and, after a standard covering argument we reach the conclusion that $V_{\tilde{\mu},q}(Dv)\in W^{1,2}_{loc}(B_{r},\mathbb{R}^{n})$. An easy computation than shows that
\begin{flalign}\label{49}
\left|D\right.&\left.(V_{\tilde{\mu},q}(Dv))\right |^{2}=\left(\frac{q-2}{2}\right)^{2}(\tilde{\mu}^{2}+\snr{Dv}^{2})^{\frac{q-6}{2}}\snr{Dv\cdot D^{2}v}^{2}\nonumber \\
&+(\tilde{\mu}^{2}+\snr{Dv}^{2})^{\frac{q-2}{2}}\snr{D^{2}v}^{2}+(q-2)(\tilde{\mu}^{2}+\snr{Dv}^{2})^{\frac{q-4}{2}}\snr{Dv\cdot D^{2}v}^{2}\nonumber \\
\ge &\min\{1,q-1\}(\tilde{\mu}^{2}+\snr{Dv}^{2})^{\frac{q-2}{2}}\snr{D^{2}v}^{2},
\end{flalign}
thus, for any given open subset $U\Subset B_{r}$, there holds that
\begin{flalign*}
\int_{U}&\snr{D^{2}v}^{2} \ \dx \le \max\left\{\tilde{\mu}^{2-q},\left[\tilde{\mu}^{2}+\nr{Dv}^{2}_{L^{\infty}(U)}\right]^{\frac{2-q}{2}}\right\}\int_{U}(\tilde{\mu}^{2}+\snr{Dv}^{2})^{\frac{q-2}{2}}\snr{D^{2}v}^{2} \ \dx\nonumber \\
&\qquad\qquad\quad\stackrel{\eqref{49}}{\le}\frac{\max\left\{\tilde{\mu}^{2-q},\left[\tilde{\mu}^{2}+\nr{Dv}^{2}_{L^{\infty}(U)}\right]^{\frac{2-q}{2}}\right\}}{\min\{1,q-1\}}\int_{U}\snr{D(V_{\tilde{\mu},q}(Dv))}^{2} \ \dx.
\end{flalign*}
Hence, after a standard covering argument, we can conclude that $v\in W^{2,2}_{loc}(B_{r})$ and, since
\begin{flalign*}
\int_{U}&\snr{D_{x_{j}}(\partial_{z_{k}}F(x,Dv))}^{2} \ \dx=\int_{U}\left| \ \partial^{2}_{x_{j},z_{k}}F(x,Dv)+\sum_{s=1}^{n}\partial_{z_{s},z_{k}}^{2}F(x,Dv)D^{2}_{x_{j},x_{s}}v\ \right |^{2} \ \dx\nonumber \\
&\quad \stackrel{\eqref{asss}_{5,6}}{\le}c\left[\htt(\nr{Dv}_{L^{\infty}(U)})^{p-1}+\htt(\nr{Dv}_{L^{\infty}(U)})^{q-1}\right]\nonumber \\
&\qquad+c\max\left\{\left[\htt(\nr{Dv}_{L^{\infty}(U)})^{p-2}+\htt(\nr{Dv}_{L^{\infty}(U)})^{q-2}\right],\tilde{\mu}^{p-2}+\tilde{\mu}^{q-2}\right\}\int_{U}\snr{D^{2}v}^{2} \ \dx,
\end{flalign*}
with $c=c(n,\tilde{\nu},\tilde{L},p,q)$, it also follows that $\partial_{z}F(\cdot,Dv)\in W^{1,2}_{loc}(\Omega,\mathbb{R}^{n})$ and we are done.
\end{proof}

\end{document}